\documentclass[psamsfonts]{amsart}

\usepackage{fullpage}
\usepackage{float}

\usepackage[parfill]{parskip}
\usepackage{placeins}

\usepackage{cite}

\usepackage[colorlinks=false,pagebackref]{hyperref}
\hypersetup{colorlinks=false, pdfborder={0 0 0}}

\usepackage{amsfonts,amssymb,amscd,amsmath,enumerate,verbatim,}
\usepackage{amssymb}

\usepackage[all]{xy}
\SelectTips{eu}{} 
\xyoption{curve}

\usepackage{tikz}
\usetikzlibrary{matrix,arrows,decorations.pathmorphing, chains, positioning,scopes}

\usepackage{young}
\usepackage[vcentermath]{youngtab}

\usepackage{eufrak}
\newcommand{\fraka}{\mathfrak a}
\newcommand{\frakm}{\mathfrak m}

\newcommand{\scrptE}{\mathcal E}
\newcommand{\scrptF}{\mathcal F}

\newcommand{\scrptH}{\mathcal H}
\newcommand{\scrptI}{\mathcal I}
\newcommand{\scrptJ}{\mathcal J}
\newcommand{\scrptK}{\mathcal K}
\newcommand{\scrptO}{\mathcal O}
\newcommand{\scrptQ}{\mathcal Q}
\newcommand{\scrptR}{\mathcal R}

\newcommand{\scrptT}{\mathcal T}
\newcommand{\scrptV}{\mathcal V}

\newcommand{\bbP}{\mathbb P}

\newcommand{\bbS}{\mathbb S}

\newcommand{\bbZ}{\mathbb Z}

\def\into{{\rightarrowtail}}
\def\onto{\twoheadrightarrow}

\newcommand{\pseudofrac}[3]{%
  \begin{subarray}{l}#1\\#2\\#3\end{subarray}%
}

\DeclareMathOperator{\Hom}{Hom}
\DeclareMathOperator{\End}{End}

\newcommand{\iso}{\simeq}
\newcommand{\niso}{\cong}

\DeclareMathOperator{\D}{D}
\DeclareMathOperator{\R}{R}
\DeclareMathOperator{\Ext}{Ext}
\DeclareMathOperator{\Tor}{Tor}

\DeclareMathOperator{\coker}{coker}
\DeclareMathOperator{\im}{im}
\DeclareMathOperator{\drsum}{\oplus}
\DeclareMathOperator{\Drsum}{\bigoplus}
\DeclareMathOperator{\tot}{tot}

\DeclareMathOperator{\tnsr}{\otimes}
\DeclareMathOperator{\TnsrAlg}{\bigotimes}

\DeclareMathOperator{\ExtAlg}{\bigwedge}
\DeclareMathOperator{\wdg}{\wedge}
\DeclareMathOperator{\Sym}{Sym}
\DeclareMathOperator{\GL}{GL}

\DeclareMathOperator{\fmod}{mod}

\DeclareMathOperator{\Spec}{Spec}
\DeclareMathOperator{\Proj}{Proj}

\DeclareMathOperator{\height}{ht}

\theoremstyle{definition}
\newtheorem{defn}{Definition}[section]
\newtheorem{definition}[defn]{Definition}

\newtheorem{remark}[defn]{Remark}

\newtheorem{example}[defn]{Example}

\theoremstyle{plain}

\newtheorem{proposition}[defn]{Proposition}

\newtheorem{theorem}[defn]{Theorem}
\newtheorem*{theorem*}{Theorem}

\newtheorem{lemma}[defn]{Lemma}

\theoremstyle{remark}

\begin{document}

\title[On Cohomology of Complexes Associated with a Generic Matrix]{On Cohomology of Complexes Associated with a Generic Matrix}


\author[Author]{Mikhail Gudim} 
\address{} \email{}


\date{\today} 

\maketitle

\begin{abstract}

In the appendix of the famous book "Commutative Algebra with a View Towards Algebraic Geometry" one can find an infinite family of complexes indexed by integers. This family includes Eagon-Northcott and Buschsbaum-Rim complexes. The objective of this paper is to study this family, and, in particular, refine the knowledge of its cohomology.

First, we obtain these complexes from the derived images of twists of the Koszul complex on the projective space. This idea  apparently goes back to Kempf [1970]. Taking this "geometric" point of view, we interpret the cohomology of these complexes as the cohomology of certain vector bundles on the projective space, and proceed with calculations. 

Finally, we put the above results in the realm of tilting theory: non-exactness of this family in certain regions can be seen as a failure of the exceptional sequence of line bundles on the projective space to lift to an exceptional sequence on a certain vector bundle. This observation creates a curious contrast with the results of Buchweitz-Leushke-Van den Bergh, stating that the exceptional sequence of twisted differential forms does lift to an exceptional sequence on the same vector bundle.

\end{abstract}

\makeatletter
\@starttoc{toc}
\makeatother

\section{Introduction}
In this paper the main object of study is the family of complexes $D^{\bullet}(i)$ (see Definition~\ref{def:D}) and their cohomology.  It was noted by Buchsbaum and Eisenbud in~\cite{Buch-Eis} that Eagon-Northcott and Buchsbaum-Rim complexes fit  into the family $D^{\bullet}(i)$. The same complexes were also discovered independently by David Kirby (\cite{Kirb}). See Section A2.6 in~\cite{Eis} for the full story. Recently, the family $D^{\bullet}(i)$ was generalized to multilinear setting in \cite{Berk-Erm-Kum-Sam}.

After necessary preparations (Section~\ref{sec:tech}), we use the idea of Kempf (from his unpublished thesis) to construct the complexes $D^{\bullet}(i)$ as the pushforward of Koszul complexes on the total space of a trivial vector bundle over the projective space. This approach immediately allows us to prove the acyclicity properties of the complexes $D^{\bullet}(i)$ (Theorem~\ref{thm:main}). Most of the parts, if not all, of Theorem~\ref{thm:main} were probably known, but were either scattered across the literature or existed in the form of mathematical folklore. In any case, we think it is a good idea to have all these results gathered in one place. Perhaps, it will be interesting to the reader to compare this construction of complexes $D^{\bullet}(i)$ with Section A2.6 in~\cite{Eis} where these complexes are described in purely algebraic manner. 

Our geometric point of view allows to interpret the cohomology of $D^{\bullet}(i)$ as the cohomology of a certain vector bundle. In Section~\ref{sec:StrRep} we calculate this cohomology using the (corollary of) Borell-Bott-Weil theorem to get the answer as a graded representation. This calculation does not reveal the structure of the cohomology as an $R$-module, however still gives us some new results (Proposition~\ref{prop:coh1}), most important of which is that the cohomology has to be non-zero, whenever allowed by Theorem~\ref{thm:main}.

In Section~\ref{sec:EquivStr} we use similar technique as in~\cite{Gud} to attempt to understand the module structure of the cohomology modules better. In particular, we can describe the lattice of cohomology modules (Theorem \ref{thm:lattice}). This is in some sence an answer, but it is far from satisfactory, because we could not obtain the resolutions, only lower bound on the projective dimension (Proposition~\ref{prop:ProjDim}).

We conclude by viewing the whole story from tilting theory perspective in Section~\ref{sec:Tilt}. After very brief overview of the necessary notions, we see that non-exactness of the family $D^{\bullet}(i)$ is the obstruction for exceptional collection of line bundles on the projective space to lift to an exceptional collection on the total space of vector bundle. However, Theorem C in~\cite{Buch-Leu-Van-den-Bergh} shows that the exceptional collection consisting of differential on the projective space does lift to an exceptional collection on the same vector bundle. As the two collections are related by a series of mutations, it is natural to ask (and this is what Ragnar-Olaf Buchweitz asked me): after which step in mutating the sequence of differential forms on the projective space the lift of the resulting sequence to the vector bundle fails to be exceptional? As it turns out, the failure already happens at the first step. Moreover, we display the two-parameter family of complexes $D^{\bullet}(i,k)$ which includes the family $D^{\bullet}(i)$. If one adopts the paradigm that the family  $D^{\bullet}(i)$ corresponds to the exceptional sequnce of line bundles on the projective space, then, roughly speaking, the family $D^{\bullet}(i,k)$ corresponds to the exceptional sequences on the projective space obtained at different stages of mutation process.

I would like to thank Ragnar-Olaf Buchweitz for pointing out this problem to me, explaining the Kempf's method and helpful suggestions. Also, I would like to thank Steven Sam for helpful suggestions. 

\section{Technical Results}
\label{sec:tech}

\subsection{Notation and Conventions}
Let $\Bbbk$ be an algebraically closed field of characteristic $0$. Let $V$ and $W$ be finite-dimensional vector spaces over $\Bbbk$ and fix bases $(e_1, \cdots e_f)$ for $V$ and $(x_1, \cdots x_g)$ for $W$. Throughout this paper we assume that $f \geq g$. We will think of elements of $V$ and $W$ as having degree $1$. Further, let $R:= \Sym (\Hom (V, W)^*) \niso \Sym (V \tnsr W^*) \iso \Bbbk [\{t_{ij}\}]$ (the last isomorphism is $e_j \tnsr x_i^* \mapsto t_{ij}$), $F:= R \tnsr V$, $G:= R \tnsr W$, $S:=\Sym G$ and let $H:= \GL(V) \times \GL (W^*)$.

In Subsections~\ref{subsec:Coass},~\ref{subsec:Rep},~\ref{subsec:Pieri} and~\ref{subsec:Emb} we work with just one $\Bbbk$-vector space $U$ of dimension $n$. We will often abbreviate $\bbS_{\lambda} U$ just as $\bbS_{\lambda}$. Let $\lambda=(\lambda_1, \lambda_2, \cdots \lambda_m)$ be some partition. By $\ExtAlg^{\lambda}V$ we will mean the tensor product $\ExtAlg^{\lambda_1}V \tnsr \ExtAlg^{\lambda_2}V \tnsr \cdots \ExtAlg^{\lambda_m}V$ and similarly for the symmetric powers.

\subsection{Coassociativity of the Exterior Algebra}
\label{subsec:Coass}
See Section I of \cite{Ak-Buch-Wey} for details of this section. The exterior algebra $\ExtAlg U$ is a graded-commutative Hopf algebra. Let $l:=(l_1, \cdots l_t)$, $a:=(a_1, \cdots a_t)$ and $b:=(b_1, \cdots b_t)$ be vectors with all integer components. Consider the map

\[
\Phi(l,a) \colon \ExtAlg^{l}U= \ExtAlg^{l_1}U \tnsr \cdots \ExtAlg^{l_t}U \to (\ExtAlg^{l_1-a_1}U \tnsr \cdots \ExtAlg^{l_t-a_t}U) \tnsr (\ExtAlg^{a_1}U \tnsr \cdots \ExtAlg^{a_t}U)=\ExtAlg^{l-a}U \tnsr \ExtAlg^{a}U 
\]

The map $\Phi(l,a)$ is a tensor product of the appropriate components of comultiplications followed by a permutation of factors.

As a consequence of coassociativity of $\ExtAlg U$ the following diagram commutes:

\begin{figure}[h!]

\begin{tikzpicture}
\matrix(m) [matrix of math nodes, 
row sep=4.0em, column sep=8.0em, 
text height=1.5ex, text depth=0.25ex]
{\ExtAlg^l U & \ExtAlg^{l-b}U \tnsr \ExtAlg^b U\\
 & \ExtAlg^{l-b-a}U \tnsr \ExtAlg^{a}U \tnsr \ExtAlg^bU \\
\ExtAlg^{l-a}U \tnsr \ExtAlg^aU & \ExtAlg^{l-a-b}U \tnsr \ExtAlg^bU \tnsr \ExtAlg^aU \\};

\path[->,font=\scriptsize] 
(m-1-1) edge node[above]{$\Phi(l,b)$}(m-1-2)
(m-3-1) edge node[above]{$\Phi(l-a,b) \tnsr 1_{\ExtAlg^aU}$}(m-3-2)
(m-1-1) edge node[right]{$\Phi(l,a)$}(m-3-1)
(m-1-2) edge node[right]{$\Phi(l-b,a) \tnsr 1_{\ExtAlg^bU}$} (m-2-2) 
(m-3-2) edge node[right]{$\sigma$} node[left]{$\niso$}(m-2-2);
\end{tikzpicture}
\caption{}
\label{fig:coass}
\end{figure}
\FloatBarrier

The map $\sigma$ is a permutation of factors.

\subsection{Representations of $\GL(U)$}
\label{subsec:Rep}

Every finite-dimensional polynomial (more generally, rational) representation of $\GL(U)$ decomposes into direct sum of irreducible ones. Irreducible polynomial representations of degree $d$ of $\GL(U)$ are indexed by partitions $\lambda$ of $d$. The irreducible representation corresponding to partition $\lambda$ is called the "Schur functor corresponding to $\lambda$ applied to $U$" and is denoted by $\bbS_{\lambda}U$ which we will simply abbreviate as $\bbS_{\lambda}$   See Section 8.2 of ~\cite{Fult} for a reference.

Let $\lambda:=(\lambda_1, \cdots \lambda_n)$ be a partition. By $D(\lambda)$ we will mean the Young diagram of $\lambda$ which we draw as a set of boxes in the plane. For example, the Young diagram of the partition $(3,2,2,1)$ is

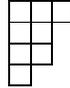
\[
{\scriptsize\young(\hfil \hfil \hfil,\hfil \hfil,\hfil \hfil,\hfil)}
\]

The conjugate partition of $\lambda$ is the partition $\tilde{\lambda}=(\tilde{\lambda}_1, \cdots \tilde{\lambda}_t)$ where $\tilde{\lambda}_i$ is the number of boxes in the $i$-th column of $D(\lambda)$. For example, the conjugate partition of $(3,2,2,1)$ is the partition $(4,3,1.0)$ and its Young diagram is:

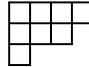
\[
{\scriptsize\young(\hfil \hfil \hfil \hfil,\hfil \hfil \hfil,\hfil)}
\]

Let us recall:

\begin{theorem}
There is only one representation occurring in both $\Sym_{\lambda} U$ and $\ExtAlg^{\tilde{\lambda}} U$. This common representation is the Schur functor $\bbS_{\lambda}$ and can be realized 

\begin{enumerate}
\item As an image of the (unique) map $\Sym_{\lambda} U \to \underset{(i,j) \in D(\lambda)}{\TnsrAlg} U(i,j) \to \ExtAlg^{\tilde{\lambda}} U$ and thus as a quotient of $\Sym_{\lambda} U$, or

\item as an image of the (unique) map $\ExtAlg^{\tilde{\lambda}} U \to \underset{(i,j) \in D(\lambda)}{\TnsrAlg} U(i,j) \to \Sym_{\lambda} U$ and thus as a quotient of $\ExtAlg^{\tilde{\lambda}} U$
\end{enumerate}

Here $U(i,j)$ means the copy of $U$ indexed by the box $(i,j)$ of $D(\lambda)$. The first map in (1) is the tensor product of components of comultiplications in the symmetric algebra: $\Sym_{\lambda_i}U \to \underset{1 \leq j \leq \lambda_i}{\TnsrAlg}U(i,j)$ and the second map is the tensor product of projections $\underset{1 \leq i \leq \tilde{\lambda}_j}{\TnsrAlg}U(i,j) \to \ExtAlg^{\tilde{\lambda}_j}U$
\end{theorem}

If we choose a basis for $U$ say $u_1, \cdots u_n$ then for each semistandard tableaux $T$ one can write down an element $u_T$ of $\ExtAlg^{\tilde{\lambda}}U$ which is a tensor product of wedge products of elements, the $i$-th factor being the wedge product of basis elements (in order) occuring in the $i$-th column of $T$. For example, the standard tableaux {\scriptsize\young(113,24)} defines the element $\pseudofrac{u_1}{\wdg}{u_2} \tnsr \pseudofrac{u_1}{\wdg}{u_4} \tnsr u_3$ of $\ExtAlg^{(2,2,1)}U$

These elements are linearly independent and span the subspace of $\ExtAlg^{\tilde{\lambda}}U$ isomorphic to $\bbS_{\lambda}$. Moreover, the images of those elements under the composite map in (2) above form a basis for the subspace of $\Sym_{\lambda}U$ isomorphic to  $\bbS_{\lambda}$.

According to our grading conventions, $\bbS_{\lambda}$ is a graded vector space concentrated in the degree $|\lambda|$- the number of boxes in $D(\lambda)$. For example, both $\bbS_{(3,2,2,1)}$ and $\bbS_{(4,3,1.0)}$ are of degree $3+2+2+1=4+3+1+0=8$.

For any vector space $U$ we adopt the convention that $U^{\tnsr m} = (U^*)^{\tnsr (-m)}$. Any irreducible rational, but not polynomial representation of $\GL(U)$ is of the form $\bbS_{\lambda} \tnsr (\ExtAlg^n U)^{\tnsr m}$ for some $m<0$. (see \cite{Wey} Proposition 2.2.1, Theorem 2.2.9 and Theorem 2.2.10). Now with any non-increasing sequence $(\lambda_1, \cdots \lambda_n)$ (where $\lambda_i$'s do not have to be non-negative) we can associate an irreducible \emph{rational} representation of $U$ by extending the rule $\bbS_{(\lambda_1, \cdots \lambda_n)} \tnsr (\ExtAlg^n U)^{\tnsr m} \niso \bbS_{(\lambda_1+m, \cdots \lambda_n+m)}$ to all values $m \in \bbZ$.

We will also use the canonical isomorphism (see Exercise 2.19 in \cite{Wey})

\[
\bbS_{(\lambda_1, \cdots \lambda_n)}(U^*) \niso \bbS_{(-\lambda_n, \cdots -\lambda_1)} U.
\]  

where $(\lambda_1, \cdots \lambda_n)$ is any non-increasing sequence of integers.

\subsection{Pieri Inclusions}
\label{subsec:Pieri}

For a partition $\lambda$ define $VS(\lambda, k)$ (here "$VS$"  stands for "vertical strip") to be the set of all partitions obtained from $\lambda$ by adjoining $k$ boxes, no two in the same row. Similarly, let $HS(\lambda, k)$ (here "$HS$" stands for "horizontal strip") to be the set of all partitions obtained from $\lambda$ by adjoining $k$ boxes with no two in the same column. Let us recall the well-known Pieri Formulas (see Section 6.1 of \cite{Fult-Harr}):

\[
\bbS_{\lambda} \tnsr \Sym_k U \niso \underset{\eta \in HS(\lambda, k)}{\Drsum} \bbS_{\eta}
\] 

and

\[
\bbS_{\lambda} \tnsr \ExtAlg^k U \niso \underset{\eta \in VS(\lambda, k)}{\Drsum} \bbS_{\eta}
\] 

For a fixed $\lambda$ and $k$ let us take $\eta \in VS(\lambda, k)$. Suppose $a_j$ boxes were added to the $j$-th column of $\lambda$ to get $\eta$ (in particular, the $a_j$'s add up to $k$) Let $a:=(a_1, \cdots a_j)$. The inclusion $\bbS_{\eta} \into \bbS_{\lambda} \tnsr \ExtAlg^k U$ is the composition (see proof of Theorem IV.2.1 in \cite{Ak-Buch-Wey})

\[
\bbS_{\eta}  \into \ExtAlg^{\tilde{\eta}}U \to \ExtAlg^{\tilde{\lambda}}U \tnsr (\ExtAlg^{a} U) \to \ExtAlg^{\tilde{\lambda}}U \tnsr \ExtAlg^k U \to \bbS_{\lambda} \tnsr \ExtAlg^kU
\]

\begin{example}
\label{ex:Pieri1}
The formula for the Pieri inclusion $\bbS_{(2,2)} \into \Sym_2U\tnsr \Sym_2U$ is:

\begin{eqnarray*}
{\scriptsize\young(ab,cd)} \mapsto \pseudofrac{a}{\wdg}{c} \tnsr \pseudofrac{b}{\wdg}{d} \mapsto (\pseudofrac{a}{\tnsr}{c} - \pseudofrac{c}{\tnsr}{a}) \tnsr (\pseudofrac{b}{\tnsr}{d} - \pseudofrac{d}{\tnsr}{b}) \mapsto \\
\pseudofrac{a}{\tnsr}{c} \tnsr \pseudofrac{b}{\tnsr}{d}-\pseudofrac{a}{\tnsr}{c} \tnsr \pseudofrac{d}{\tnsr}{b}-\pseudofrac{c}{\tnsr}{a} \tnsr \pseudofrac{b}{\tnsr}{d}+\pseudofrac{c}{\tnsr}{a} \tnsr \pseudofrac{d}{\tnsr}{b} \mapsto \\
ab \tnsr cd-ad \tnsr cb-cb \tnsr ad+cd \tnsr ab
\end{eqnarray*}

\end{example}

\begin{example}
\label{ex:Pieri2}
The formula for the Pieri inclusion $\bbS_{(2,1)} \into U \tnsr \Sym_2U$ is:

\begin{eqnarray*}
{\scriptsize\young(ab,c)} \mapsto \pseudofrac{a}{\wdg}{c} \tnsr b \mapsto \pseudofrac{a}{\tnsr}{c} \tnsr b - \pseudofrac{c}{\tnsr}{a} \tnsr b \mapsto c \tnsr ab - a \tnsr cb
\end{eqnarray*}

\end{example}

\subsection{Embeddings into $U \tnsr \bbS_{\lambda} \tnsr \ExtAlg^kU$}
\label{subsec:Emb}

We will need a technical result which describes the embeddings of irreducible representations into $U \tnsr \bbS_{\lambda} \tnsr \ExtAlg^k U$. First let us think about this tensor product with brackets placed as follows: $U \tnsr (\bbS_{\lambda} \tnsr \ExtAlg^k U)$. With this placement of brackets we will think about each irreducible summand $\eta$ as being obtained from the diagram of $\lambda$ by first adding $k$ boxes according to Pieri rule (i.e. first tensor with $\ExtAlg^k U$) and then adding one more box (tensor with $U$). Moreover, we will mark the $k$ boxes that were added first with the symbol "$\ExtAlg$" and the box which was added last will be marked with the symbol "$U$". 

\begin{example}

Suppose the dimension of $U$ is $3$. Consider the tensor product $U \tnsr \bbS_{(2,1,0)} \tnsr \ExtAlg^2 U$. The irreducible summands of $\bbS_{(2,1,0)} \tnsr \ExtAlg^2 U$ are obtained from the diagram of $\lambda$ by adding two boxes according to the Pieri rule. Let us mark these boxes with the symbol '$\ExtAlg$'. Thus the summands of $\bbS_{(2,1,0)} \tnsr \ExtAlg^2 U$ are:

{\scriptsize\young( \hfil \hfil \ExtAlg,\hfil \ExtAlg)}, {\scriptsize\young( \hfil \hfil \ExtAlg,\hfil,\ExtAlg)}, {\scriptsize\young( \hfil \hfil,\hfil \ExtAlg,\ExtAlg)}

Now to get all the summands of $U \tnsr (\bbS_{(2,1,0)} \tnsr \ExtAlg^2 V)$ we add one more box to all the summands above. Let us mark this box with the symbol "$U$". We get:

{\scriptsize\young( \hfil \hfil \ExtAlg U,\hfil \ExtAlg)}, {\scriptsize\young( \hfil \hfil \ExtAlg,\hfil \ExtAlg U)}, {\scriptsize\young( \hfil \hfil \ExtAlg,\hfil \ExtAlg,U)},\\

{\scriptsize\young( \hfil \hfil \ExtAlg U,\hfil,\ExtAlg)}, {\scriptsize\young( \hfil \hfil \ExtAlg,\hfil U,\ExtAlg)}, {\scriptsize\young( \hfil \hfil \ExtAlg,\hfil,\ExtAlg,U)},\\

{\scriptsize\young( \hfil \hfil U,\hfil \ExtAlg,\ExtAlg)}, {\scriptsize\young( \hfil \hfil,\hfil \ExtAlg,\ExtAlg U)}, {\scriptsize\young( \hfil \hfil,\hfil \ExtAlg,\ExtAlg,U)}.
\end{example}

Note that with this bracket placement no "$\ExtAlg$"'s will be to the right or below the "$U$". Each such (labeled) diagram $\eta$ defines an embedding $\bbS_{\eta} \into U \tnsr (\bbS_{\lambda}  \tnsr \ExtAlg^k U)$ in the following way: suppose $\eta$ has $a_j$ $\ExtAlg$'s in the column $j$ and the $U$ is in the column $s$. Let $a:=(a_1, \cdots a_t)$ and $b$  be the vector all of whose components are zero, except $1$ in the position $s$.

\begin{eqnarray*}
\bbS_{\eta} \into \ExtAlg^{\tilde{\eta}} U \xrightarrow{\Phi(\tilde{\eta}, b)} U \tnsr \ExtAlg^{\tilde{\eta}-b}U \xrightarrow{1_U \tnsr \Phi(\tilde{\eta}-b,a)} U \tnsr (\ExtAlg^{\tilde{\eta}-b-a}U \tnsr \ExtAlg^aU) \to U \tnsr (\bbS_{\lambda} \tnsr \ExtAlg^kU)
\end{eqnarray*}

In other words, this labeling of $\eta$ tells us to first separate off the box labeled "$U$" by applying appropriate component of comultiplication on exterior algebra. Then we separate off all the "$\ExtAlg$"'s. Precomposing with the inclusion $\bbS_{\eta} \into \ExtAlg^{\tilde{\eta}} U$ and postcomposing with the projection gives us the desired embedding.

Now let us think about this tensor product with brackets placed in a different way: $(U \tnsr \bbS_{\lambda}) \tnsr \ExtAlg^k U$. With this placement of the brackets we will think about each irreducible summand $\eta'$ as being obtained from the diagram of $\lambda$ by first adding one box (i.e. first tensor with $U$) and then adding $k$ more boxes (tensor with $\ExtAlg^k U$). Again, we will mark the box that was added first with the symbol "$U$" and the $k$ boxes which were added later will be marked with the symbol "$\ExtAlg$". 

\begin{example}
Now we think about the tensor product $U \tnsr \bbS_{(2,1,0)} \tnsr \ExtAlg^2 U$ with brackets placed as $(U \tnsr \bbS_{(2,1,0)}) \tnsr \ExtAlg^2 U$. The summands of $U \tnsr \bbS_{(2,1,0)}$ are:

{\scriptsize\young( \hfil \hfil U,\hfil)}, {\scriptsize\young( \hfil \hfil,\hfil U)}, {\scriptsize\young( \hfil \hfil,\hfil,U)}

Now we add two "$\ExtAlg$"'s according to the Pieri rule:

{\scriptsize\young( \hfil \hfil U\ExtAlg,\hfil \ExtAlg)}, {\scriptsize\young( \hfil \hfil U\ExtAlg,\hfil,\ExtAlg)}, {\scriptsize\young( \hfil \hfil U,\hfil \ExtAlg,\ExtAlg)}, \\

{\scriptsize\young( \hfil \hfil \ExtAlg,\hfil U\ExtAlg)}, {\scriptsize\young( \hfil \hfil \ExtAlg,\hfil U,\ExtAlg)},\\

{\scriptsize\young( \hfil \hfil \ExtAlg,\hfil \ExtAlg,U)}, {\scriptsize\young( \hfil \hfil \ExtAlg,\hfil,U,\ExtAlg)}, {\scriptsize\young( \hfil \hfil,\hfil \ExtAlg,U\ExtAlg)}, {\scriptsize\young( \hfil \hfil,\hfil \ExtAlg,U,\ExtAlg)}

\end{example}

Note that this time "$U$" is "inside" of "$\ExtAlg$"'s, i.e. all the "$\ExtAlg$" occur to the right and below the "$U$". Each such (labeled) diagram $\eta'$ defines an embedding $\bbS_{\eta'} \into (U \tnsr \bbS_{\lambda}) \tnsr \ExtAlg^k U$ in the following way: suppose $\eta$ has $a'_j$ $\ExtAlg$'s in the column $j$ and the $U$ is in the column $s'$. Let $a':=(a'_1, \cdots a'_t)$ and $b'$  be the vector all of whose components are zero, except $1$ in the position $s'$.

\begin{eqnarray*}
\bbS_{\eta'} \into \ExtAlg^{\tilde{\eta'}} U \xrightarrow{\Phi(\tilde{\eta'}, a')} \ExtAlg^{\tilde{\eta'}-a'}U \tnsr \ExtAlg^{a'}U \xrightarrow{\Phi(\tilde{\eta'}-a',b') \tnsr 1_{\ExtAlg^{a'}U}} (U \tnsr \ExtAlg^{\tilde{\eta'}-a'-b'}U)\tnsr \ExtAlg^{a'}U \to U \tnsr (\bbS_{\lambda} \tnsr \ExtAlg^kU)
\end{eqnarray*}

We will need to know "how to move "$U$" from outside to inside". More precisely, suppose we have a labeled diagram $\eta$ which comes from the bracket placement $U \tnsr (\bbS_{\lambda} \tnsr \ExtAlg^k U)$. This labeled diagram defines the embedding $\bbS_{\eta} \into U \tnsr (\bbS_{\lambda} \tnsr \ExtAlg^k U)$. The question is: which labbeled diagram $\eta'$ which comes from bracket placement $(U \tnsr \bbS_{\lambda}) \tnsr \ExtAlg^kU$ defines the same embedding as $\eta$? We will write $\eta \sim \eta'$ when $\eta$ and $\eta'$ define the same embedding into $U \tnsr \bbS_{\lambda} \tnsr \ExtAlg^k U$. 

The skew shape $\eta-\lambda$ has at most one row with two boxes. Suppose $\eta-\lambda$ has a row with two boxes, then multiplicity of such $\eta$ in $U \tnsr \bbS_{\lambda} \tnsr \ExtAlg^k U$ is equal to $1$, so there is only one labeled diagram $\nu$ coming from bracket placement $(U \tnsr \bbS_{\lambda}) \tnsr \ExtAlg^k U$ that has the same shape as $\eta$.

\begin{example}
In the above example we have the following equivalences:

{\scriptsize\young( \hfil \hfil \ExtAlg U,\hfil \ExtAlg)} $\sim$ {\scriptsize\young( \hfil \hfil U\ExtAlg,\hfil \ExtAlg)}, {\scriptsize\young( \hfil \hfil \ExtAlg,\hfil \ExtAlg U)} $\sim$ {\scriptsize\young( \hfil \hfil \ExtAlg,\hfil U\ExtAlg)}, {\scriptsize\young( \hfil \hfil \ExtAlg U,\hfil,\ExtAlg)} $\sim$ {\scriptsize\young( \hfil \hfil U\ExtAlg,\hfil,\ExtAlg)},{\scriptsize\young( \hfil \hfil,\hfil \ExtAlg,\ExtAlg U)} $\sim$ {\scriptsize\young( \hfil \hfil,\hfil \ExtAlg,U\ExtAlg)}. 

\end{example}

Suppose now that the skew-shape $\eta-\lambda$ has no row with two boxes. If there is no column that has both "$\ExtAlg$"'s and "$U$" then this $\eta$ appears in both bracket placements and no moving is necessary.

\begin{example}
In our example we see that the labeled diagrams

{\scriptsize\young( \hfil \hfil \ExtAlg,\hfil \ExtAlg,U)}, {\scriptsize\young( \hfil \hfil \ExtAlg,\hfil U,\ExtAlg)}, {\scriptsize\young( \hfil \hfil \ExtAlg,\hfil \ExtAlg,U)}.

appear in both bracket placements.
\end{example}

Finally, suppose that the skew shape $\eta-\lambda$ has no row with two boxes and some column has both "$\ExtAlg$"'s and "$U$". Then we need to "slide the "$U$" up". This does not change the embedding. One can easily see it by writing down the corresponding embeddings and using coassotiativity of $\ExtAlg U$ (see Figure ~\ref{fig:coass}).

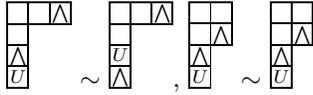
\begin{example}
{\scriptsize\young( \hfil \hfil \ExtAlg,\hfil,\ExtAlg,U)} $\sim$ {\scriptsize\young( \hfil \hfil \ExtAlg,\hfil,U,\ExtAlg)}, {\scriptsize\young( \hfil \hfil,\hfil \ExtAlg,\ExtAlg,U)} $\sim$ {\scriptsize\young( \hfil \hfil,\hfil \ExtAlg,\ExtAlg,U)}.
\end{example}

In a similar way we can prove analogous results for embeddings into $U \tnsr \bbS_{\lambda} \tnsr \Sym_k U$.

\subsection{Representations of $H:=\GL(V) \times \GL(W^*)$}
\label{subsec:RepH}
Every finite-dimensional polynomial (more generally, rational) representation of $H$ decomposes into direct sum of irreducible ones. Every irreducible polynomial representation of $H$ is of the form $\bbS_{\lambda} V \tnsr \bbS_{\eta} W^*$ (see Exercise 2.36 of~\cite{Fult-Harr}). 

We will also need the so-called Cauchy formula for the decomposition of $\Sym (V \tnsr W^*)$ into irreducible representations of $H$ (see Section III.1 of~\cite{Ak-Buch-Wey}):

\begin{equation}
\Sym_d (V \tnsr W^*) \niso \underset{|\lambda|=d}{\Drsum} \bbS_{\lambda} V \tnsr \bbS_{\lambda} W^*
\end{equation}

The formula for the embedding $\bbS_{\lambda} V \tnsr \bbS_{\lambda} W^*$ is as follows. First define a map 

\begin{eqnarray*}
\alpha_k \colon \ExtAlg^k V \tnsr \ExtAlg^k W^* \to \Sym_k (V \tnsr W^*),\\
\alpha_k(v_1 \wdg \cdots v_k \tnsr w^*_1 \wdg \cdots w^*_k):= \det
\begin{bmatrix}
    v_1 \tnsr w^*_1 & \cdots & v_1 \tnsr w^*_k\\

    &  & \vdots &  & \\
   v_k \tnsr w^*_1 & \cdots & v_k \tnsr w^*_k\\
  \end{bmatrix} 
\end{eqnarray*}

Suppose now $\lambda=(\lambda_1, \cdots \lambda_t)$ is any partition of $d$, then the inclusion $\bbS_{\lambda} V \tnsr \bbS_{\lambda} W^* \into \Sym_d (V \tnsr W^*)$ is the following composition:

\begin{eqnarray*}
\bbS_{\lambda}V \tnsr \bbS_{\lambda}W^* \into \ExtAlg^{\lambda} V \tnsr \ExtAlg^{\lambda}W^* = (\ExtAlg^{\lambda_1} V \tnsr \cdots \ExtAlg^{\lambda_t}V) \tnsr (\ExtAlg^{\lambda_1}W^* \tnsr \cdots \ExtAlg^{\lambda_t}W^*) \niso \\
(\ExtAlg^{\lambda_1}V \tnsr \ExtAlg^{\lambda_1}W^*) \tnsr \cdots (\ExtAlg^{\lambda_t}V \tnsr \ExtAlg^{\lambda_t}W^*) \to \Sym_{\lambda_1}(V \tnsr W^*) \tnsr \cdots \Sym_{\lambda_t}(V \tnsr W^*) \to \Sym_d (V \tnsr W^*)
\end{eqnarray*}

where the first map is the tensor product of the inclusions of the $\bbS_{\lambda}$ into $\ExtAlg^{\lambda}$,  the third map is the tensor product of $\alpha_{\lambda_i}$'s and the last map is multiplication on the symmetric algebra.

\subsection{Vector Bundles and Locally Free Sheaves}
There are several choices how to setup the correspondence between vector bundles over a scheme $S$ and locally free sheaves on $S$ (see Exercises II.5.17-II.5.18 of \cite{Har} for details). We will stick to the following convention: Let $\scrptE$ be a locally free sheaf over $S$. Then the corresponding vector bundle $X \xrightarrow{f} S$ (here $f \colon X \to S$ denotes the structure map) to $\scrptE$ is $\underline{\Spec} \Sym \scrptE^*$. Let $\scrptJ(X)$ denote the sheaf of sections of $X$. Then, according to our conventions, $\scrptJ(X)=\scrptJ(\underline{\Spec} \Sym \scrptE^*)\niso\scrptE$ and $f_* \scrptO_X \niso \Sym \scrptF^*$.    By "cohomology of the vector bundle $X$" we will mean the cohomology of its sheaf of sections.

Also, we apply the same convention to projective spaces, so $\bbP(\scrptE):= \Proj \Sym \scrptE^*$.
Note that our conventions are different from \cite{Har} for example, but are the same as in \cite{Wey}.

\begin{example}
Let $\bbP:=\bbP(W^*)=\underline{\Proj} \Sym W$ be the projective space on $W^*$. There is the so-called (exact) tautological sequence of vector bundles on $\bbP$:

\begin{equation}
\label{eq:taut}
0 \to \scrptR \to W^* \times \bbP \to \scrptQ \to 0
\end{equation}

Here $\scrptR$ is the tautological rank $1$ subbundle of the trivial bundle $\scrptR:=\{(w^*,[L]) \in W^* \times \bbP | w^* \in L\}$ and $\scrptQ$ denotes the quotient. Then $\scrptR \niso \underline{\Spec} \Sym (\scrptO(1))$ and $\scrptQ \niso \underline{\Spec} \Sym (\Omega(1))$.

\end{example}

Since applying Schur functor to a vector space is a functorial operation, it extends to vector bundles in a natural way. By a Schur functor $\bbS_{\lambda}$ applied to a locally free sheaf $\scrptE$ we will mean the sheaf associated to $\bbS_{\lambda}(\underline{\Spec} \Sym \scrptE^*)$.

\subsection{Cohomology of Projective Spaces} Again, let $S$ be a scheme, $\scrptE$ a locally free sheaf of rank $r$, $\bbP(\scrptE^*):= \underline{\Proj} \Sym \scrptE$ and let $p: \bbP(\scrptE^*) \to S$ the structure map. Then:

\[
\R^q p_*\scrptO_{\bbP(\scrptE^*)}(d) =
\begin{cases}
\Sym_d(\scrptE) \text{ if } q=0 \\
0 \text{ if } q \neq 0,r-1 \\
\scrptH om_{\scrptO_S}(\Sym_{-d-r} \scrptE \tnsr \ExtAlg^r \scrptE, \scrptO_S) \text{ if } q=r-1 
\end{cases}
\]

We will also use the following formulation of Serre Duality on projetive space: there is an isomorphism \emph{of functors} 

\[
(\R^ip_*(\_))^* \niso \R^{r-1-i}((\_)^* \tnsr p^* \ExtAlg^r \scrptE (-r))
\]

\subsection{Borel-Bott-Weil Theorem for Projective Spaces}

The sheaves $\bbS_{\lambda} (\scrptT(-1))(i)$ and $\bbS_{\lambda} (\Omega(1))(i)$ correspond to homogeneous vector bundles on projective space and their cohomology is known and can be derived as a corollary of Borel-Bott-Weil theorem. Let us recall this result (which is Corollary 4.1.9 from \cite{Wey} in case of projective space, slightly rephrased). For $\sigma$ any permutation of $g$ letters let $l(\sigma)$ denote its length - the minimal number of transpositions (of two adjacent letters) necessary to express $\sigma$. The action of $\sigma$ on any sequence $(n_1, \cdots n_g) \in \bbZ^g$ is defined as $\sigma (n_1, \cdots n_g):=(n_{\sigma(1)}, \cdots n_{\sigma(g)})$.

\begin{theorem}
\label{thm:BBW}

Let $\bbP:= \underline{\Proj} \Sym W$ and $p \colon \bbP \to \Spec \Bbbk$ be the structure map. Let $i \in \bbZ$, $\lambda=(\lambda_1, \cdots , \lambda_{g-1}), \lambda_1 \geq \cdots \geq \lambda_{g-1}$ be an non-incresing sequence of integers (negative values are allowed). Consider the vector bundle $\scrptV(i, \lambda):= (\scrptR^*)^{\tnsr i} \tnsr \bbS_{\lambda} \scrptQ^*$ (Recall that sheaf of sections of $\scrptR^*$ is $\scrptO(1)$ and sheaf of sections of $\scrptQ^*$ is $\Omega(1)$). Then $\R^qp_* \scrptV(i, \lambda)$ is non-zero for at most one $q$ and

\begin{enumerate}

\item If the sequence $(i+g-1, \lambda_1+g-2, \cdots \lambda_{g-1})$ has a repetition, then $\scrptV(i, \lambda)$ has no cohomology, i.e. $\R^qp_* \scrptV(i, \lambda) = \R^qp_* \bbS_{\lambda}(\Omega(1))(i)=0$ for all $q$. 

\item If the sequence $(i+g-1, \lambda_1+g-2, \cdots \lambda_{g-1})$ has no repetitions, then there exists a unique permutation $\sigma$ such that the the sequence $\delta (\lambda):=\sigma(i+g-1, \lambda_1+g-2, \cdots \lambda_{g-1})-(g-1, g-2, \cdots, 0)$ is non-increasing. In this case $\scrptV(i, \lambda)$ has only one cohomology which is

\[
\R^{l(\sigma)}p_* \scrptV(i, \lambda) = \R^{l(\sigma)}p_* \bbS_{\lambda}(\Omega(1))(i)=\bbS_{\delta(\lambda)} W
\]
\end{enumerate}

\end{theorem} 

\begin{example}
For a warm-up, let us calculate cohomology of the sheaves $\scrptO(d)$ on projective space $\bbP = \Proj \Sym W$. In the theorem above we set $\lambda=(0, \cdots 0)$. Consider the sequence $(d+g-1, g-2, \cdots 0)$. This sequence has a repetition if and only if $-g+1 \leq d \leq -1$ in which case the sheaf $\scrptO(d)$ has no cohomology. So assume $d$ is not in this range. The sequence $\delta (\lambda):=\sigma(d+g-1, g-2, \cdots 0)-(g-1, g-2, \cdots, 0)$ is non-increasing if and only if the sequence $\sigma(d+g-1, g-2, \cdots 0)$ is \emph{strictly} decreasing. To achieve that we only possibly need to move $d+g-1$ across smaller numbers. 

If $d \geq 0$ then no moving is necessary. In this case $\scrptO(d)$ only has $0$-th cohomology which is $\bbS_{\delta} W = \bbS_{(d, 0, \cdots 0)} W = \Sym_d W$. 

If $d<-g+1$, then the number $d+g-1$ is negative and thus we need to move it to the very right of the sequence to make it strictly decreasing. The length of this permutation is $g-1$ and so $\scrptO(d)$ only has $(g-1)$-st cohomology which is $\bbS_{\delta(\lambda)} W = \bbS_{(-1, \cdots -1, d+g-1)} W \niso \bbS_{(0, \cdots 0, d+g)} W \tnsr \bbS_{(-1, \cdots -1)} W \niso \Sym_{-d-g} W^* \tnsr \ExtAlg^g W^*$. 
\end{example}

\begin{example}
\label{ex:coh}
Fix $i<0$ and $q \geq 0$. In this example we calculate for which partitions $\lambda=(\lambda_1, \cdots \lambda_{g-1}), \lambda_1 \geq \lambda_2 \geq \cdots \lambda_{g-1} \geq 0$ the cohomology $\R^q p_*\bbS_{\lambda} (\scrptT(-1))(i)$ is non-zero. This will be used later in Section~\ref{sec:StrRep}.

The sheaf $\bbS_{\lambda} (\scrptT(-1))(i)$ corresponds to the vector bundle $\scrptV(i, \lambda)=\scrptR^{\tnsr (-i)} \tnsr \bbS_{\lambda} \scrptQ \niso (\scrptR^*)^{\tnsr i} \tnsr \bbS_{(-\lambda_{g-1}, -\lambda_{g-2}, \cdots -\lambda_1)} \scrptQ^*$. We need to find all partitions $\lambda$ such that the sequence $\sigma(i+g-1, -\lambda_{g-1}+g-2, \cdots -\lambda_1)$ is strictly decreasing where $\sigma$ is the (unique) permuation of length $q$. The subsequence $(-\lambda_{g-1}+g-2, \cdots -\lambda_1)$ is strictily decreasing already, so the partition $\lambda$ should be such that the number $i+g-1$ has to be moved to the right across exactly $q$ numbers:  $-\lambda_{g-1}+g-2, \cdots -\lambda_{g-q}+g-q-1$. This happens if and only if the inequalities $-\lambda_{g-q-1}+g-q-2 <i+g-1 < -\lambda_{g-q}+g-q-1$ hold. This is equivalent to 

\begin{equation}
\label{eq:lambda}
\lambda_{-q+g}<-i-q \leq \lambda_{-q+g-1}
\end{equation}

If $\lambda$ is a partition satisfying inequalities~\ref{eq:lambda} then, 

\begin{eqnarray*}
\R^qp_* \scrptV(i, \lambda)=\R^qp_* \scrptR^{\tnsr (-i)} \tnsr \bbS_{\lambda} \scrptQ \niso \bbS_{\delta(\lambda)} W = \bbS_{(-\lambda_{g-1}-1, \cdots \lambda_{g-q}-1, i+q, -\lambda_{g-q-1}, \cdots -\lambda_1)}W \niso \\
\bbS_{(\lambda_1, \cdots \lambda_{g-q-1}, -i-q, \lambda_{g-q}+1, \cdots \lambda_{g-1}+1)}W^* \niso \bbS_{(\lambda_1-1, \cdots \lambda_{g-q-1}-1, -i-q-1, \lambda_{g-q}, \cdots \lambda_{g-1})}W^* \tnsr \ExtAlg^g W^*
\end{eqnarray*}

Note that if $\lambda$ is a partition satisfying inequalities~\ref{eq:lambda}, then $\delta(\lambda)$ can be described as follows: remove one box from each of the first $-q+g-1$ rows of $\lambda$, add one box to the bottom of each of the first $-i-q-1$ columns and then add a column of $g$ boxes.
\end{example}

\subsection{Spectral Sequences}
 
Let $C^{\bullet}$ be any complex with filtration:

\[
F \colon \cdots \to F_2 C^{\bullet} \to F_1 C^{\bullet} \to F_0 C^{\bullet}:=C^{\bullet}
\]

This filtration induces a filtration on cohomology of $C^{\bullet}$ in the obvious way by $H^{\bullet}F_k:= \im H^{\bullet}(F_k C^{\bullet} \into C^{\bullet})$. Suppose now that we have a double complex $D^{\bullet, \bullet}$ with differentials $d_{vert} \colon D^{i,j} \to D^{i, j+1}$ and $d_{hor} \colon D^{i,j} \to D^{i+1, j}$. In our applications the double complex will be concentrated in the second quadrant in a strip of finite width, meaning that $D^{p, \bullet}=0$ for $p>0$ and for $p<<0$, and $D^{\bullet, q}=0$ for $q<0$. 

Let $C^{\bullet}:= \tot D^{\bullet, \bullet}$ be the total complex of $D^{\bullet, \bullet}$. There are two natural filtrations on $C^{\bullet}$: one is $F^{col}$ which is induced from filtering $D^{\bullet, \bullet}$ by columns, and the other, $F^{row}$, is induced from filtering by rows.  

\begin{theorem}

With the notation above, the filtration $F^{col}$ of $C^{\bullet}$ gives rise to the spectral sequence $^{col}E_{\bullet}^{\bullet, \bullet}$ with

\[
^{col}E_1^{p,q} \niso H^q(D^{p, \bullet}) 
\]

and

\[
^{col}E_2^{p,q} \niso H^p(\cdots \to H^q(D^{i-1, \bullet}) \to H^q(D^{i, \bullet}) \to H^q(D^{i+1, \bullet}) \to \cdots)
\]

The differntial $d_r$ on the $r-$th page 

\[
d_r \colon ^{col}E_r^{p,q} \to E_r^{p+r, q-r+1}
\]

Similarly, the filtration $F^{row}$ of $C^{\bullet}$ gives rise to the spectral sequence $^{row}E_{\bullet}^{\bullet, \bullet}$ with

\[
^{row}E_1^{p,q} \niso H^p(D^{\bullet, q}) 
\]

and

\[
^{row}E_2^{p,q} \niso H^q(\cdots \to H^p(D^{\bullet, i-1}) \to H^p(D^{\bullet, i}) \to H^p(D^{\bullet, i+1}) \to \cdots)
\]

The differntial $d_r$ on the $r-$th page 

\[
d_r \colon ^{row}E_r^{p,q} \to E_r^{p-r+1, q+r}
\]

Both spectral sequences converge, meaning that $\underset{p+q=n}{\Drsum} ^{col}E_{\infty}^{p,q}$ is the associated graded of $H^nC^{\bullet}$ with respect to the filtration induced by $F^{col}$ on $H^nC^{\bullet}$ and similarly for the spectral sequence $^{row}E_{\bullet}^{\bullet, \bullet}$.
\end{theorem}

\section{The Geometric Setup}

\subsection{A General Construction}
\label{subsec:GenConst}
 In this subsection let $\phi \colon F \to G$ be \emph{any} linear map. We can perform the following construction. Let $\bbP := \bbP_R G^* = \Proj S$ and $p \colon \bbP \to \Spec R$ be the structure map. Then $G \niso p_* \scrptO_{\bbP}(1)$. Apply the chain of isomorphisms 

\[
\Hom_R(F, p_* \scrptO_{\bbP}(1)) \niso \Hom_{\bbP}(p^*F, \scrptO_{\bbP}(1)) \niso \Hom_{\bbP}(p^*F(-1), \scrptO_{\bbP})
\]

to $\phi$ to get $\phi'$. Locally on $U_i:= \Spec \scrptO_{U_i} := \Spec R[ \frac{x_1}{x_i}, \cdots \widehat{\frac{x_i}{x_i}} \cdots ,\frac{x_n}{x_i} ] \hookrightarrow \bbP$, $\scrptO_{\bbP}(-1)$ is generated by $\{\frac{1}{x_i}\}$ and $\phi'$ does the following: $m\tnsr t\{\frac{1}{x_i}\} \mapsto \frac{\phi(m)}{x_i}t$ for $m \in F, t \in \scrptO_{U_i}$. Let $Z$ denote the zero locus of $\phi'$.
The defining homogeneous ideal of $Z$ is generated by the linear forms $\phi(e_1), \cdots ,\phi(e_f)$.

Now consider the Koszul complex (of sheaves) constructed from $\phi'$:

\begin{figure}[h!]

\begin{tikzpicture}
\matrix(m) [matrix of math nodes, 
row sep=0.5em, column sep=0.5em, 
text height=1.5ex, text depth=0.25ex]
{\scrptK^{\bullet}(\phi) \colon 0 & \underset{-f}{(\ExtAlg^f p^*F)(-f)} & (\ExtAlg^{f-1} p^*F)(-f+1) & \cdots & (\ExtAlg^2 p^*F)(-2) & p^*F(-1) & \underset{0}{\scrptO_{\bbP}} & 0\\};

\path[->,font=\scriptsize] 
(m-1-1) edge (m-1-2)
(m-1-2) edge (m-1-3) 
(m-1-3) edge (m-1-4) 
(m-1-4) edge (m-1-5)
(m-1-5) edge (m-1-6)
(m-1-6) edge (m-1-7)
(m-1-7) edge (m-1-8);
\end{tikzpicture}
\end{figure}

\subsection{The Generic Case} Let us apply the construction of the previous subsection when $\phi$ is the generic map, i.e. the matrix of  $\phi$ is $\phi_{ij}=t_{ij}$. Let $\bbP(W^*):=\Proj \Sym W$, $\bbP:=\bbP_R(G^*):= \Proj S \niso \bbP(W^*) \times \Spec R$ - the total space of trivial vector bundle (of dimension $fg$) over $\bbP(W^*)$. We think of the fiber over a point $[\lambda] \in \bbP(W^*)$ as $\Hom(V,W)$. Recall that $Z$ denotes the intersection of zero locii of $R-$linear forms $\phi(e_1), \cdots \phi(e_f)$ on $\bbP$, i.e. 

\[
Z:= \Proj_R S/<\phi(e_1), \cdots \phi(e_f)>.
\]

This setup is summarized in the diagram:
 
\begin{figure}[h!]
\label{fig:GeomDiag}
\begin{tikzpicture}
\matrix(m) [matrix of math nodes, 
row sep=4.0em, column sep=4.0em, 
text height=1.5ex, text depth=0.25ex]
{Z & \Spec R/ I_g\\
\bbP & \Spec R \\
\bbP(W^*) & \Spec \Bbbk\\};

\path[->,font=\scriptsize] 
(m-1-1) edge(m-1-2)
(m-2-1) edge node[above]{$p$}(m-2-2)
(m-3-1) edge node[above]{$p'$}(m-3-2)
(m-1-1) edge node[right]{$\iota$}(m-2-1)
(m-1-2) edge(m-2-2)
(m-2-1) edge node[right]{$\pi$}(m-3-1)
(m-2-2) edge node[right]{$\pi'$}(m-3-2);
\end{tikzpicture}
\caption{}
\label{fig:setup}
\end{figure}
\FloatBarrier

We now show that $Z$ is a (non-trivial) subbundle of the trivial vector bundle and the fiber over $[\lambda] \in \bbP(W^*)$ is $\{a \in \Hom(V, W) | \lambda \in \ker a^*\}$, i.e the vector space of linear maps from $V$ to $W$ whose image is contained in $\ker \lambda$. 

Consider the set $\tilde{Z}:=\{(a, \lambda) \in \Hom(V,W) \times W^*\ | \lambda \in \ker a^*\}$. We will show that this an affine scheme and find its equations.  For that let us look at the evaluation map $ev \colon \Hom(V,W) \times W^* \to V^*$ given by $ev(a, \lambda):= a^*(\lambda)$. The equations for $\tilde{Z}$ are obtained by calculating the fiber of $ev$ above $0 \in V^*$.

After the identification $\Hom(V,W) \niso  V^* \tnsr W$ the action of $ev$ on basis elements is:

\[
ev(e_j^* \tnsr x_i, x_k^*)=(e_l \mapsto x_k^*(e_j^*(e_l)x_i)=(e_l \mapsto \delta_{ki} \delta_{jl})	
\]

The ring map corresponding to $ev$ is:

\[
\Sym V \to R \tnsr \Sym W \niso S \colon e_l \mapsto \sum_{i,j,k} \delta_{ki} \delta_{jl} e_j \tnsr x_i^* \tnsr x_k = \sum_i e_l \tnsr x_i^* \tnsr x_i	
\]

and the image of $e_j$ in $S$ is $\phi(e_j)$. Equations of the fiber are calculated by the tensor product:

\[
S \underset{\Sym V}{\tnsr} (\Sym V / <e_1, \cdots e_f>) \niso S/<\phi(e_1), \cdots \phi(e_f)> = S/<\im \phi>
\]

Now let us verify that the composition $Z \xrightarrow{\iota} \bbP \niso \bbP(W^*) \times \Spec R \xrightarrow{\pi} \bbP(W^*)$ of inclusion $\iota$ and projection $\pi$ is a vector bundle. Locally on the standard affine open set $U_k:= \Spec \scrptO_{U_k} := \Spec \Bbbk [ \frac{x_1}{x_k}, \cdots \widehat{\frac{x_k}{x_k}} \cdots ,\frac{x_n}{x_k} ] \into \bbP$ our map is obtained as follows. Look at the composition

\[
\Bbbk[x_1, \cdots x_g] \to R[x_1, \cdots x_g] \to R[x_1, \cdots x_g]/< \sum_i t_{i1} x_i, \cdots \sum_i t_{if} x_i> 
\]

then localize at $x_k$ and take the zero-th graded component with respect to grading by $x_i$'s to get

\[
\Bbbk[\frac{x_1}{x_k}, \cdots \widehat{\frac{x_k}{x_k}} \cdots ,\frac{x_g}{x_k}] \to 
R[\frac{x_1}{x_k}, \cdots \widehat{\frac{x_k}{x_k}} \cdots ,\frac{x_g}{x_k}] \to 
R[\frac{x_1}{x_k}, \cdots \widehat{\frac{x_k}{x_k}} \cdots ,\frac{x_g}{x_k}]/< (t_{k1} + \sum_{i \neq k} t_{i1}\frac{x_i}{x_k}), \cdots (t_{kf} + \sum_{i \neq k} t_{if}\frac{x_k}{x_i})>
\]

But now each of $f$ relations removes the variable $t_{kj}$, so the quotient on the right is just the polynomial ring $\Bbbk [\{t_{ij}\}_{i \neq k}][\frac{x_1}{x_k}, \cdots \widehat{\frac{x_k}{x_k}} \cdots ,\frac{x_n}{x_k}]$ in $fg-f+(g-1)=(f+1)(g-1)$ variables. Thus indeed, locally our map is just a projection and each fiber is a vector space of dimension $f(g-1)$ as it should be.

The affine morphism $\pi \iota \colon Z \to \bbP(W^*)$ defines quasi-coherent (locally free) sheaf of algebras on $\bbP(W^*)$, namely $(\pi \iota)_* \scrptO_Z$, so $Z \niso \underline{\Spec} (\pi \iota)_* \scrptO_Z$. We now express this in more familiar terms. We will identify the graded $\Sym W$-module which corresponds to the sheaf $(\pi\iota)_* \scrptO_Z$.

Recall that we have the following sequence of graded $R$- modules:

\[
S \tnsr F \xrightarrow{\phi'} S \to S/< \im \phi> \to 0
\]

The graded $\Sym W$-module corresponding to the sheaf $(\pi \iota)_* \scrptO_Z$ is exactly $S/< \im \phi>$ regarded as $\Sym W$-module, but recall that as $\Sym W$-modules we have:

\[
S=\Sym G =\Sym W \underset{\Bbbk}{\tnsr} \Sym (V \tnsr W^*) \niso \Sym_{\Sym W} (\Sym W \tnsr (V \tnsr W^*)) 
\]

And  $\coker \phi'$ is just the algebra $\Sym G$ modulo the ideal generated by some linear forms, so as $\Sym W-$modules (see Proposition A2.2(d) of \cite{Eis}):

\begin{equation}
\label{eq:coker}
S/< \im \phi> \niso \Sym_{\Sym W} (\coker ( \Sym W \tnsr V \to \Sym W \tnsr (V \tnsr W^*)))
\end{equation}

On the other hand, take the Euler sequence on $\bbP(W^*)$:

\[
0 \to \Omega(1) \to \Sym W \tnsr W \to \Sym W \to 0 
\]

Now dualize this sequence (apply $\Hom_{\Sym W}(\_, \Sym W)$), tensor (over $\Sym W$) with $\Sym W \tnsr V$ and observe that $\scrptT(-1) \tnsr V = \coker ( \Sym W \tnsr V \to \Sym W \tnsr (V \tnsr W^*))$ is the cokernel of the same map as in~(\ref{eq:coker}), thus 

\begin{equation}
(\pi \iota)_* \scrptO_Z \niso \Sym (\scrptT(-1) \tnsr p'^* V)
\end{equation}

Moreover, we have the following identification of vector bundles:

\[
Z \niso \scrptQ^* \tnsr V^*
\]

which can be easilly obtained from dualizing the tautological sequence \ref{eq:taut} and tensoring with $V^*$.

\begin{theorem}
When $\phi$ is the generic map, the Koszul complex of sheaves $\scrptK^{\bullet}(\phi)$ is exact and resolves $\scrptO_Z$
\end{theorem}

\begin{proof}
We need to check that the sequence of elements $\phi(e_1)=\sum_{l=1}^g t_{l1}x_l, \cdots \phi(e_f)=\sum_{l=1}^g t_{lf}x_l$ is a regular sequence in the polynomial ring $\Bbbk[\{t_{ij}\}][x_1 ,\cdots x_g]$ (in $(fg + g)$ variables!) when localized at each maximal ideal $\frakm$ not containing all of $x_1, \cdots x_g$. Suppose say $x_1 \not\in \frakm$. Let us check the definition of regular sequence. First, $\phi(e_1)$ is clearly not a zero divisor. Equation of $\phi(e_1)$ allows us to write $t_{11}$ as a linear combination of other $t$'s ($x_1$ is invertible), so, taking quotient just removes $t_{11}$. Expression of $\phi(e_2)$ does not even involve $t_{11}$, so $\phi(e_2)$ is clearly a nonzero divisor on the quotient. And so on.
\end{proof}

It is also interesting to calculate  the fibers of the other projection $Z \xrightarrow{\iota} \bbP(W^*) \times \Spec R \xrightarrow{p} \Spec R$. Let $\fraka \in \Spec R$ be a maximal ideal corresponding to the matrix with entries $a_{ij}$, $\kappa(\fraka):= R/\fraka$ and $\Spec \kappa(\fraka) \to \Spec R$ inclusion of the point. Then the fiber over $\fraka$ is just the fiber product over $\Spec R$ of $Z$ and $\Spec \kappa(\fraka)$, which is $\Proj_{\kappa(\fraka)} \coker (\kappa(\fraka) \tnsr \phi) \iso \Proj_{\Bbbk} \Sym ((\coker a)^*) \niso \Proj_{\Bbbk} \Sym (\ker a^*)$. In particular, if the rank of $a$ is the maximum possible (equal to $g$) then the fiber above is empty and conversely. This means that the image is exactly the zero locus of $g \times g$ minors, i.e. $\Spec R/I_g$ where $I_g$ is the ideal generated by $g \times g$ minors.

\section{The Complexes $D^{\bullet}(i)$}
\label{sec:GenComp}

We will abbreviate the complex $\scrptK^{\bullet}(\phi)$ from the previous section just as $\scrptK^{\bullet}$. For $i \in \bbZ$ by $\scrptK^{\bullet}(i)$ we mean the complex $\scrptK^{\bullet} \tnsr \scrptO_{\bbP}(i)$ which resolves $\scrptO_Z(i)$. Moreover, let us introduce the following notation:

\begin{definition}
Let $i \in \bbZ$. We define the complexes of $R$-modules:
\[
K^{\bullet}(i):=\R^0p_*\scrptK^{\bullet}(i)
\]

and 
\[
C^{\bullet}(i):= \R^{g-1}p_* \scrptK^{\bullet}(i-g)
\]

The complex $K^{\bullet}(i)$ is supported in the interval $[max(-i,-f), 0]$ and the complex $C^{\bullet}(i)$ is supported in the interval $[-f, min(0,-i)]$
\end{definition}

Note that

\begin{eqnarray*}
C^{\bullet}(f-i)=\R^{g-1}p_* \scrptK^{\bullet}(f-g-i) \niso \R^{g-1}p_*(\scrptH om (\scrptK^{\bullet}(g+i), \scrptO) \tnsr p^*\ExtAlg^gG \tnsr p^*\ExtAlg^fF) \tnsr \ExtAlg^gG^*[f] \niso \\
\R^{g-1}p_* (\scrptH om (\scrptK^{\bullet}(i), \scrptO) \tnsr p^* \ExtAlg^gG(-g)) \tnsr \ExtAlg^fF \tnsr \ExtAlg^gG^*[f] \niso \Hom_R(\R^0p_* \scrptK^{\bullet}(i), R) \tnsr \ExtAlg^fF \tnsr \ExtAlg^gG^* [f] \niso \\
\Hom_R(K^{\bullet}(i), R) \tnsr \ExtAlg^fF \tnsr \ExtAlg^gG^* [f]
\end{eqnarray*}

The first isomorphism is the duality of the Koszul complex $\scrptH om (\scrptK^{\bullet}(j), \scrptO)[f] \niso \scrptK^{\bullet}(f-j) \tnsr p^* \ExtAlg^fF^*$ and the third isomorphism is the Serre duality on the projective space.

Next two diagrams show the complexes $K^{\bullet}(i)$ (Figure~\ref{fig:K}) and $C^{\bullet}(i)$ (Figure~\ref{fig:C}). We abbreviate $\Sym_i G$ as $S_i$. Also for typographical reasons we omit the twist by $\ExtAlg^g G^*$ in complexes $C^{\bullet}(i)$

\begin{figure}

\begin{tikzpicture}[scale=1.4]

\node at (1,0) {\scalebox{0.7}{$0$}};
\node at (0,0) {\scalebox{0.7}{$R$}};
\node at (-1,0) {\scalebox{0.7}{$0$}};

\node at (1,1) {\scalebox{0.7}{$0$}};
\node at (0,1) {\scalebox{0.7}{$G$}};
\node at (-1,1) {\scalebox{0.7}{$F$}};
\node at (-2,1) {\scalebox{0.7}{$0$}};

\node at (1,2) {\scalebox{0.7}{$0$}};
\node at (0,2) {\scalebox{0.7}{$S_2$}};
\node at (-1,2) {\scalebox{0.7}{$S_1 \tnsr F$}};
\node at (-2,2) {\scalebox{0.7}{$\ExtAlg^2 F$}};
\node at (-3,2) {\scalebox{0.7}{$0$}};

\node at (1,3) {\scalebox{0.7}{$0$}};
\node at (0,3) {\scalebox{0.7}{$S_3$}};
\node at (-1,3) {\scalebox{0.7}{$S_2 \tnsr F$}};
\node at (-2,3) {\scalebox{0.7}{$S_1 \tnsr \ExtAlg^2 F$}};
\node at (-3,3) {\scalebox{0.7}{$\ExtAlg^3 F$}};
\node at (-4,3) {\scalebox{0.7}{$0$}};

\node at  (0,5) {$\vdots$};

\node at (1,7) {\scalebox{0.7}{$0$}};
\node at (0,7) {\scalebox{0.7}{$S_i$}};
\node at (-1,7) {\scalebox{0.7}{$S_{i-1} \tnsr F$}};
\node at (-2,7) {\scalebox{0.7}{$\cdots$}};
\node at (-3,7) {\scalebox{0.7}{$\cdots$}};
\node at (-4,7) {\scalebox{0.7}{$S_{i-k} \tnsr \ExtAlg^k F$}};
\node at (-5,7) {\scalebox{0.7}{$\cdots$}};
\node at (-6,7) {\scalebox{0.7}{$S_1 \tnsr \ExtAlg^{i-1} F$}};
\node at (-7,7) {\scalebox{0.7}{$\ExtAlg^i F$}};
\node at (-8,7) {\scalebox{0.7}{$0$}};

\node at  (0,9) {$\vdots$};

\node at (1,10) {\scalebox{0.7}{$0$}};
\node at (0,10) {\scalebox{0.7}{$S_{f-2}$}};
\node at (-1,10) {\scalebox{0.7}{$S_{f-3} \tnsr F$}};
\node at (-2,10) {\scalebox{0.7}{$\cdots$}};
\node at (-8,10) {\scalebox{0.7}{$\cdots$}};
\node at (-9,10) {\scalebox{0.7}{$S_1 \tnsr \ExtAlg^{f-3} F$}};
\node at (-10,10) {\scalebox{0.7}{$\ExtAlg^{f-2} F$}};
\node at (-11,10) {\scalebox{0.7}{$0$}};

\node at (1,11) {\scalebox{0.7}{$0$}};
\node at (0,11) {\scalebox{0.7}{$S_{f-1}$}};
\node at (-1,11) {\scalebox{0.7}{$S_{f-2} \tnsr F$}};
\node at (-2,11) {\scalebox{0.7}{$\cdots$}};
\node at (-9,11) {\scalebox{0.7}{$\cdots$}};
\node at (-10,11) {\scalebox{0.7}{$S_1 \tnsr \ExtAlg^{f-2} F$}};
\node at (-11,11) {\scalebox{0.7}{$\ExtAlg^{f-1} F$}};
\node at (-12,11) {\scalebox{0.7}{$0$}};

\node at (1,12) {\scalebox{0.7}{$0$}};
\node at (0,12) {\scalebox{0.7}{$S_{f}$}};
\node at (-1,12) {\scalebox{0.7}{$S_{f-1} \tnsr F$}};
\node at (-2,12) {\scalebox{0.7}{$\cdots$}};
\node at (-10,12) {\scalebox{0.7}{$\cdots$}};
\node at (-11,12) {\scalebox{0.7}{$S_1 \tnsr \ExtAlg^{f-1} F$}};
\node at (-12,12) {\scalebox{0.7}{$\ExtAlg^{f} F$}};
\node at (-13,12) {\scalebox{0.7}{$0$}};

\node at (1,13) {\scalebox{0.7}{$0$}};
\node at (0,13) {\scalebox{0.7}{$S_{f+1}$}};
\node at (-1,13) {\scalebox{0.7}{$S_{f} \tnsr F$}};
\node at (-2,13) {\scalebox{0.7}{$\cdots$}};
\node at (-10,13) {\scalebox{0.7}{$\cdots$}};
\node at (-11,13) {\scalebox{0.7}{$S_2 \tnsr \ExtAlg^{f} F$}};
\node at (-12,13) {\scalebox{0.7}{$S_1 \tnsr \ExtAlg^{f} F$}};
\node at (-13,13) {\scalebox{0.7}{$0$}};

\node at (1,14) {\scalebox{0.7}{$0$}};
\node at (0,14) {\scalebox{0.7}{$S_{f+2}$}};
\node at (-1,14) {\scalebox{0.7}{$S_{f+1} \tnsr F$}};
\node at (-2,14) {\scalebox{0.7}{$\cdots$}};
\node at (-10,14) {\scalebox{0.7}{$\cdots$}};
\node at (-11,14) {\scalebox{0.7}{$S_1 \tnsr \ExtAlg^{f-1} F$}};
\node at (-12,14) {\scalebox{0.7}{$S_2 \tnsr \ExtAlg^{f} F$}};
\node at (-13,14) {\scalebox{0.7}{$0$}};

\path[->,font=\scriptsize] 

(-1,0) edge[shorten <= .5cm, shorten >= .5cm] (0,0) 

(-1,1) edge[shorten <= .5cm, shorten >= .5cm] (0,1) 
(-2,1) edge[shorten <= .5cm, shorten >= .5cm] (-1,1) 

(-2,2) edge[shorten <= .5cm, shorten >= .5cm] (-1,2) 
(-1,2) edge[shorten <= .5cm, shorten >= .5cm] (0,2)
(-3,2) edge[shorten <= .5cm, shorten >= .5cm] (-2,2) 

(-4,3) edge[shorten <= .5cm, shorten >= .5cm] (-3,3) 
(-3,3) edge[shorten <= .5cm, shorten >= .7cm] (-2,3)
(-2,3) edge[shorten <= .5cm, shorten >= .5cm] (-1,3)
(-1,3) edge[shorten <= .5cm, shorten >= .5cm] (-0,3)

(-8,7) edge[shorten <= .5cm, shorten >= .5cm] (-7,7) 
(-7,7) edge[shorten <= .5cm, shorten >= .7cm] (-6,7)
(-1,7) edge[shorten <= .5cm, shorten >= .5cm] (0,7)
(-2,7) edge[shorten <= .5cm, shorten >= .5cm] (-1,7)
(-4,7) edge[shorten <= .7cm, shorten >= .5cm] (-3,7)
(-5,7) edge[shorten <= .5cm, shorten >= .7cm] (-4,7)
(-6,7) edge[shorten <= .7cm, shorten >= .5cm] (-5,7)

(-1,10) edge[shorten <= .7cm, shorten >= .5cm] (0,10)
(-2,10) edge[shorten <= .5cm, shorten >= .5cm] (-1,10)
(-9,10) edge[shorten <= .7cm, shorten >= .5cm] (-8,10)
(-10,10) edge[shorten <= .5cm, shorten >= .7cm] (-9,10)
(-11,10) edge[shorten <= .5cm, shorten >= .5cm] (-10,10) 

(-1,11) edge[shorten <= .5cm, shorten >= .5cm] (0,11)
(-2,11) edge[shorten <= .5cm, shorten >= .5cm] (-1,11)
(-10,11) edge[shorten <= .7cm, shorten >= .5cm] (-9,11)
(-11,11) edge[shorten <= .5cm, shorten >= .7cm] (-10,11)
(-12,11) edge[shorten <= .5cm, shorten >= .5cm] (-11,11) 

(-1,12) edge[shorten <= .5cm, shorten >= .5cm] (0,12)
(-2,12) edge[shorten <= .5cm, shorten >= .5cm] (-1,12)
(-11,12) edge[shorten <= .7cm, shorten >= .5cm] (-10,12)
(-12,12) edge[shorten <= .5cm, shorten >= .7cm] (-11,12)
(-13,12) edge[shorten <= .5cm, shorten >= .5cm] (-12,12) 

(-1,13) edge[shorten <= .5cm, shorten >= .5cm] (0,13)
(-2,13) edge[shorten <= .5cm, shorten >= .5cm] (-1,13)
(-11,13) edge[shorten <= .7cm, shorten >= .5cm] (-10,13)
(-12,13) edge[shorten <= .6cm, shorten >= .6cm] (-11,13)

(-13,13) edge[shorten <= .5cm, shorten >= .7cm] (-12,13) 

(-1,14) edge[shorten <= .5cm, shorten >= .5cm] (0,14)
(-2,14) edge[shorten <= .5cm, shorten >= .5cm] (-1,14)
(-11,14) edge[shorten <= .7cm, shorten >= .5cm] (-10,14)
(-12,14) edge[shorten <= .55cm, shorten >= .7cm] (-11,14)
(-13,14) edge[shorten <= .5cm, shorten >= .7cm] (-12,14); 

\end{tikzpicture}

\caption{The complexes $K^{\bullet}(i)$}
\label{fig:K}
\end{figure}
\FloatBarrier

\begin{figure}

\begin{tikzpicture}[scale=1.6]

\node at (-12,12) {\scalebox{0.7}{$\ExtAlg^fF$}};
\node at (-13,12) {\scalebox{0.7}{$0$}};
\path[->,font=\scriptsize] 
(-13,12) edge[shorten <= .5cm, shorten >= .5cm] (-12,12);

\node at (-11,11) {\scalebox{0.7}{$\ExtAlg^{f-1}F$}};
\node at (-12,11) {\scalebox{0.7}{$\ExtAlg^f \tnsr S^*_1$}};
\node at (-13,11) {\scalebox{0.7}{$0$}};
\path[->,font=\scriptsize] 
(-13,11) edge[shorten <= .5cm, shorten >= .5cm] (-12,11)
(-12,11) edge[shorten <= .5cm, shorten >= .5cm] (-11,11);

\node at (-10,10) {\scalebox{0.7}{$\ExtAlg^{f-2}F$}};
\node at (-11,10) {\scalebox{0.7}{$\ExtAlg^{f-1}F \tnsr S^*_1$}};
\node at (-12,10) {\scalebox{0.7}{$\ExtAlg^{f}F \tnsr S^*_2$}};
\node at (-13,10) {\scalebox{0.7}{$0$}};
\path[->,font=\scriptsize] 
(-13,10) edge[shorten <= .5cm, shorten >= .7cm] (-12,10)
(-12,10) edge[shorten <= .5cm, shorten >= .7cm] (-11,10)
(-11,10) edge[shorten <= 0.7cm, shorten >= 0.5cm] (-10,10);

\node at (-13,9) {\scalebox{0.7}{$\vdots$}};

\node at (-8,8) {\scalebox{0.7}{$\ExtAlg^{f-i}F$}};
\node at (-9,8) {\scalebox{0.7}{$\ExtAlg^{f-i+1}F \tnsr S^*_1$}};
\node at (-10,8){\scalebox{0.7}{$\cdots$}};
\node at (-11,8) {\scalebox{0.7}{$\ExtAlg^{f-i+j}F \tnsr S^*_j$}};
\node at (-12,8){\scalebox{0.7}{$\cdots$}};
\node at (-13,8) {\scalebox{0.7}{$0$}};
\path[->,font=\scriptsize] 
(-13,8) edge[shorten <= .5cm, shorten >= .7cm] (-12,8)
(-12,8) edge[shorten <= .2cm, shorten >= .8cm] (-11,8)
(-11,8) edge[shorten <= 0.7cm, shorten >= 0.5cm] (-10,8)
(-10,8) edge[shorten <= 0.5cm, shorten >= 0.7cm] (-9,8)
(-9,8) edge[shorten <= 0.7cm, shorten >= 0.5cm] (-8,8);

\node at (-13,7) {\scalebox{0.7}{$\vdots$}};

\node at (-3,6) {\scalebox{0.7}{$F$}};
\node at (-4,6) {\scalebox{0.7}{$\ExtAlg^{2}F \tnsr S^*_1$}};
\node at (-5,6){\scalebox{0.7}{$\cdots$}};
\node at (-8,6){\scalebox{0.7}{$\cdots$}};
\node at (-9,6) {\scalebox{0.7}{$\ExtAlg^{f+j+1}F \tnsr S^*_j$}};
\node at (-10,6){\scalebox{0.7}{$\cdots$}};
\node at (-11,6) {\scalebox{0.7}{$\ExtAlg^{f-1}F \tnsr S^*_{f-2}$}};
\node at (-12,6) {\scalebox{0.7}{$\ExtAlg^{f}F \tnsr S^*_{f-1}$}};
\node at (-13,6) {\scalebox{0.7}{$0$}};
\path[->,font=\scriptsize] 
(-13,6) edge[shorten <= .5cm, shorten >= .7cm] (-12,6)
(-12,6) edge[shorten <= .4cm, shorten >= .8cm] (-11,6)
(-11,6) edge[shorten <= 0.7cm, shorten >= 0.5cm] (-10,6)
(-10,6) edge[shorten <= 0.5cm, shorten >= 0.7cm] (-9,6)
(-9,6) edge[shorten <= 0.7cm, shorten >= 0.5cm] (-8,6)
(-5,6) edge[shorten <= 0.5cm, shorten >= 0.7cm] (-4,6)
(-4,6) edge[shorten <= 0.5cm, shorten >= 0.5cm] (-3,6);

\node at (-2,5)  {\scalebox{0.7}{$R$}};
\node at (-3,5) {\scalebox{0.7}{$\ExtAlg^1F \tnsr S^*_1$}};
\node at (-4,5){\scalebox{0.7}{$\cdots$}};
\node at (-7,5) {\scalebox{0.7}{$\cdots$}};
\node at (-8,5) {\scalebox{0.7}{$\ExtAlg^{j}F \tnsr S^*_j$}};
\node at (-9,5) {\scalebox{0.7}{$\cdots$}};
\node at (-10,5) {\scalebox{0.7}{$\ExtAlg^{f-2}F \tnsr S^*_{f-2}$}};
\node at (-11,5) {\scalebox{0.7}{$\ExtAlg^{f-1}F \tnsr S^*_{f-1}$}};
\node at (-12,5) {\scalebox{0.7}{$\ExtAlg^{f}F \tnsr S^*_{f-1}$}};
\node at (-13,5) {\scalebox{0.7}{$0$}};
\path[->,font=\scriptsize] 
(-13,5) edge[shorten <= .5cm, shorten >= .7cm] (-12,5)
(-12,5) edge[shorten <= .4cm, shorten >= .8cm] (-11,5)
(-11,5) edge[shorten <= .5cm, shorten >= .8cm] (-10,5)
(-10,5) edge[shorten <= 0.5cm, shorten >= 0.7cm] (-9,5)
(-9,5) edge[shorten <= 0.7cm, shorten >= 0.5cm] (-8,5)
(-8,5) edge[shorten <= 0.5cm, shorten >= 0.7cm] (-7,5)
(-4,5) edge[shorten <= 0.5cm, shorten >= 0.5cm] (-3,5)
(-3,5) edge[shorten <= 0.5cm, shorten >= 0.5cm] (-2,5);

\node at (-2,4) {\scalebox{0.7}{$S^*_1$}};
\node at (-3,4) {\scalebox{0.7}{$\ExtAlg^1F \tnsr S^*_2$}};
\node at (-4,4){\scalebox{0.7}{$\cdots$}};
\node at (-7,4) {\scalebox{0.7}{$\cdots$}};
\node at (-8,4) {\scalebox{0.7}{$\ExtAlg^{j}F \tnsr S^*_{j+1}$}};
\node at (-9,4) {\scalebox{0.7}{$\cdots$}};
\node at (-10,4) {\scalebox{0.7}{$\ExtAlg^{f-2}F \tnsr S^*_{f-1}$}};
\node at (-11,4) {\scalebox{0.7}{$\ExtAlg^{f-1}F \tnsr S^*_{f}$}};
\node at (-12,4) {\scalebox{0.7}{$\ExtAlg^{f}F \tnsr S^*_{f+1}$}};
\node at (-13,4) {\scalebox{0.7}{$0$}};
\path[->,font=\scriptsize] 
(-13,4) edge[shorten <= .5cm, shorten >= .7cm] (-12,4)
(-12,4) edge[shorten <= .4cm, shorten >= .8cm] (-11,4)
(-11,4) edge[shorten <= .5cm, shorten >= .8cm] (-10,4)
(-10,4) edge[shorten <= 0.5cm, shorten >= 0.5cm] (-9,4)
(-9,4) edge[shorten <= 0.5cm, shorten >= 0.8cm] (-8,4)
(-8,4) edge[shorten <= 0.5cm, shorten >= 0.7cm] (-7,4)
(-4,4) edge[shorten <= 0.5cm, shorten >= 0.5cm] (-3,4)
(-3,4) edge[shorten <= 0.5cm, shorten >= 0.5cm] (-2,4);

\node at (-2,3) {\scalebox{0.7}{$S^*_2$}};
\node at (-3,3) {\scalebox{0.7}{$\ExtAlg^1F \tnsr S^*_3$}};
\node at (-4,3) {\scalebox{0.7}{$\cdots$}};
\node at (-7,3) {\scalebox{0.7}{$\cdots$}};
\node at (-8,3) {\scalebox{0.7}{$\ExtAlg^{j}F \tnsr S^*_{j+2}$}};
\node at (-9,3) {\scalebox{0.7}{$\cdots$}};
\node at (-10,3) {\scalebox{0.7}{$\ExtAlg^{f-2}F \tnsr S^*_{f}$}};
\node at (-11,3) {\scalebox{0.7}{$\ExtAlg^{f-1}F \tnsr S^*_{f+1}$}};
\node at (-12,3) {\scalebox{0.7}{$\ExtAlg^{f}F \tnsr S^*_{f+2}$}};
\node at (-13,3) {\scalebox{0.7}{$0$}};
\path[->,font=\scriptsize] 
(-13,3) edge[shorten <= .5cm, shorten >= .7cm] (-12,3)
(-12,3) edge[shorten <= .4cm, shorten >= .8cm] (-11,3)
(-11,3) edge[shorten <= .5cm, shorten >= .8cm] (-10,3)
(-10,3) edge[shorten <= 0.5cm, shorten >= 0.5cm] (-9,3)
(-9,3) edge[shorten <= 0.3cm, shorten >= 0.8cm] (-8,3)
(-8,3) edge[shorten <= 0.3cm, shorten >= 0.8cm] (-7,3)
(-4,3) edge[shorten <= 0.5cm, shorten >= 0.5cm] (-3,3)
(-3,3) edge[shorten <= 0.5cm, shorten >= 0.5cm] (-2,3);

\end{tikzpicture}
\caption{The complexes $C^{\bullet}(i)$ with the twist by $\ExtAlg^g G^*$ omitted}
\label{fig:C}
\end{figure}
\FloatBarrier

Now we will describe a spectral sequence which will splice the complexes $C^{\bullet}(i+g)$ and $K^{\bullet}(i)$ together. The resulting complex will be called $D^{\bullet}(i)$ (see Definition~\ref{def:D}).

Fix an $i \in \bbZ$. Let us choose a Cartan-Eilenberg resolution $\scrptK^{\bullet}(i) \to \scrptI^{\bullet, \bullet}$ for $\scrptK^{\bullet}(i)$. It will be concentrated in the strip $[-f, 0] \times [0, \infty]$. This resolution is constructed so that $H^p \scrptK^{\bullet}(i) \to (\cdots \to H^p \scrptI^{\bullet,j-1} \to H^p \scrptI^{\bullet,j} \to H^p \scrptI^{\bullet,j+1} \to \cdots)$ is an injective resolution. Now apply $p_*$:

\begin{figure}[h!]
\begin{tikzpicture}[description/.style={fill=white,inner sep=2pt}]
\matrix (m) [matrix of math nodes, row sep=1.5em,
column sep=1.5em, text height=1.5ex, text depth=0.25ex]
{& \vdots & \vdots & & \vdots & \vdots & \\
0 & p_* \scrptI^{-f,1} & p_* \scrptI^{-f+1,1} & \cdots & p_* \scrptI^{-1,1} & p_* \scrptI^{0,1} & 0\\
0 & p_* \scrptI^{-f,0} & p_* \scrptI^{-f+1,0} & \cdots & p_* \scrptI^{-1,0} & p_* \scrptI^{0,0} & 0 \\};

\path[->,font=\scriptsize] 
(m-2-6) edge (m-2-7) (m-2-5) edge(m-2-6) (m-2-4) edge (m-2-5) (m-2-3) edge (m-2-4) (m-2-2) edge (m-2-3) (m-2-1) edge (m-2-2)

(m-3-6) edge(m-3-7) (m-3-5) edge(m-3-6) (m-3-4) edge (m-3-5) (m-3-3) edge (m-3-4) (m-3-2) edge (m-3-3) (m-3-1) edge (m-3-2)

(m-3-2) edge(m-2-2) (m-3-3) edge(m-2-3) (m-3-5) edge(m-2-5) (m-3-6) edge(m-2-6)
(m-2-2) edge(m-1-2) (m-2-3) edge(m-1-3) (m-2-5) edge(m-1-5) (m-2-6) edge(m-1-6);
\end{tikzpicture}
\end{figure}
\FloatBarrier

As usual, there are two filtrations on this double complex - one by columns and one by rows, giving rise to two  spectral sequences: $_{\bullet}^{col}E_{\bullet}^{\bullet}$ and $_{\bullet}^{row}E_{\bullet}^{\bullet}$ respectively, both of them converge, because our double complex is concentrated in a vertical strip of finite width. Let us first look at $_{\bullet}^{row}E_{\bullet}^{\bullet}$:

\[
_1^{row}E^{p,q} \niso H^p(p_* \scrptI^{\bullet,q}) \niso p_*H^p \scrptI^{\bullet,q}
\]

because by definition of Cartan-Eilenberg resolution all the rows are split. So, on the first page of this spectral sequence the $p$-th column is the result of applying $p_*$ to injective resolution of  $H^p \scrptK^{\bullet}(i)$. Thus:
\[
_2^{row}E^{p,q} \niso \R^qp_*H^p \scrptK^{\bullet}(i)
\]

But recall that $\scrptK^{\bullet}(i)$ is a resolution of $\scrptO_Z(i)$, so the row spectral sequence looks as $_{\infty}^{row}E^{0,q} \niso \R p_*^q \scrptO_Z(i)$ and zero elsewhere. As there is only one nonzero term on each diagonal we see that the induced filtration on $H^{\bullet}(\tot p_* \scrptI^{\bullet, \bullet})$ is trivial, so 

\begin{equation}
\label{eq:tot}
H^n(\tot p_* \scrptI^{\bullet, \bullet}) \niso \R^n p_* \scrptO_Z(i)
\end{equation}

Now we turn to the second spectral sequence:

\begin{eqnarray*}
_1^{col}E^{p,q} \niso \R^qp_*\scrptK^p(i) \niso \R^qp_*(p^*(\ExtAlg^{-p} F) \tnsr \scrptO_{\bbP}(i+p)) \niso \R^qp_* \scrptO_{\bbP}(i+p) \tnsr \ExtAlg^{-p} F \niso \\
\begin{cases}
\Sym_{i+p} G \tnsr \ExtAlg^{-p} F & \text{if } q=0\\
\Sym^*_{-g-(i+p)} G \tnsr \ExtAlg^{-p} F \tnsr \ExtAlg^g G^* &  \text{if } q=g-1
\end{cases}
\end{eqnarray*}

The differentials are horizontal, so we see the complex $_1^{col}E^{\bullet,g-1}$ occuring in $(g-1)-$st row  concentrated in the strip $[-f, min(0,-i-g)]$ which is exactly $C^{\bullet}(i+g)$ and the complex in the $0-$th row concentrated in the strip $[max(-i, -f), 0]$ is $K^{\bullet}(i)$.

Suppose $i \leq -1$ or $f-g+1 \leq i$. In this case only one of $C^{\bullet}(i+g)$ or $K^{\bullet}(i)$ enters the picture and the spectral sequence $_{\bullet}^{col}E^{\bullet,\bullet}$ stabilizes on the second page. 

Otherwise, $C^{\bullet}(i+g)$ and $K^{\bullet}(i)$ exist simultaneously, creating possibility for nontrivial differentials. From the second page until $(g-1)$-st all differentials are zero, but there is a possibly nonzero one on page $g$! So, almost all terms stabilize on the second page, but there is a nontrivial differential 

\[
\theta_0 \colon H^{-i-g}C^{\bullet}(i+g) \to H^{-i} K^{\bullet}(i)
\]

Let 

\[
\theta \colon \ExtAlg^{i+g} F \tnsr \ExtAlg^g G^* \onto H^{-i-g}C^{\bullet}(i+g) \xrightarrow{\theta_0} H^{-i}K^{\bullet}(i) \into K^{-i}(i)=\ExtAlg^i F
\]

be the composition. Now we can splice $C^{\bullet}(i+g)$ and $K^{\bullet}(i)$ together. 

\begin{definition} 
\label{def:D}

Using the notation above for each $i \in \bbZ$ we define the complex 

\[
D^{\bullet}(i):= C^{\bullet}(i+g) \xrightarrow{\theta} K^{\bullet}(i)
\]

This definition also applies when no splicing is necessary (in which case splicing map $\theta_0$ is just zero) and $D^{\bullet}(i)$ coincides with one of $C^{\bullet}(i+g)$ or $K^{\bullet}(i)$. In this case, the cohomological grading of $D^{\bullet}(i)$ is the same as on $C^{\bullet}(i+g)$ or $K^{\bullet}(i)$. Otherwise (when the splicing map is non-zero), we define cohomological grading on $D^{\bullet}(i)$ so that the furthest non-zero term in the direction against the differential is in cohomological degree $-f$.
\end{definition}

The Figure~\ref{fig:D} shows the schematic picture of complexes $D^{\bullet}(i)$.

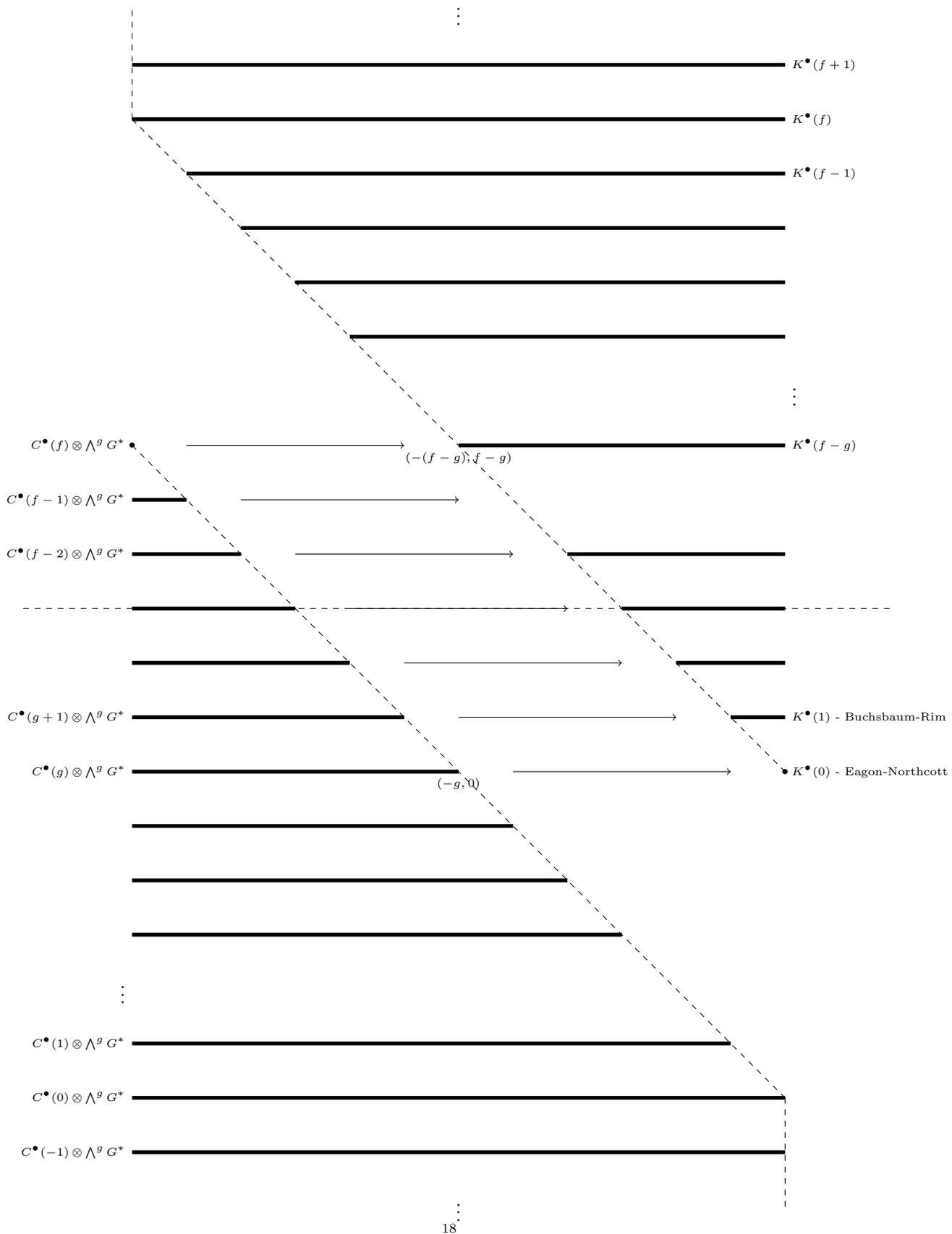
\begin{figure}
\begin{tikzpicture}[scale=1]
\draw [line width=2] (-12,13) -- (0,13);
\node[right] at (0,13){\tiny{$K^{\bullet}(f+1)$}};
\draw [line width=2] (-12,12) -- (0,12);
\node[right] at (0,12){\tiny{$K^{\bullet}(f)$}};
\draw [line width=2] (-11,11) -- (0,11);
\node[right] at (0,11){\tiny{$K^{\bullet}(f-1)$}};
\draw [line width=2] (-10,10) -- (0,10);
\draw [line width=2] (-9, 9) -- (0,9);
\draw [line width=2] (-8, 8) -- (0,8);
\node[right] at (0,7){$\vdots$};
\draw [line width=2] (-6, 6) -- (0,6);
\node[right] at (0,6){\tiny{$K^{\bullet}(f-g)$}};
\draw [line width=2] (-4, 4) -- (0,4);
\draw [line width=2] (-3, 3) -- (0,3);
\draw [line width=2] (-2, 2) -- (0,2);
\draw [line width=2] (-1, 1) -- (0,1);
\node[right] at (0,1){\tiny{$K^{\bullet}(1)$ - Buchsbaum-Rim}};
\node[right] at (0,0){\tiny{$K^{\bullet}(0)$ - Eagon-Northcott}};
\node[] at (0,0){\tiny{$\bullet$}};
\draw [dashed] (-12, 12) -- (0,0);
\draw [dashed] (-12, 12) -- (-12,14);
\node[] at (-12,6){\tiny{$\bullet$}};
\node[left] at (-12,6){\tiny{$C^{\bullet}(f) \tnsr \ExtAlg^gG^*$}};
\draw [line width=2] (-12,5) -- (-11,5);
\node[left] at (-12,5){\tiny{$C^{\bullet}(f-1) \tnsr \ExtAlg^gG^*$}};
\draw [line width=2] (-12, 4) -- (-10,4);
\node[left] at (-12,4){\tiny{$C^{\bullet}(f-2) \tnsr \ExtAlg^gG^*$}};
\draw [line width=2] (-12, 3) -- (-9,3);
\draw [line width=2] (-12, 2) -- (-8,2);
\draw [line width=2] (-12, 1) -- (-7,1);
\node[left] at (-12,1){\tiny{$C^{\bullet}(g+1) \tnsr \ExtAlg^gG^*$}};
\draw [line width=2] (-12, 0) -- (-6,0);
\node[left] at (-12,0){\tiny{$C^{\bullet}(g) \tnsr \ExtAlg^gG^*$}};
\draw [line width=2] (-12,-1) -- (-5,-1);
\draw [line width=2] (-12,-2) -- (-4,-2);
\draw [line width=2] (-12,-3) -- (-3,-3);
\node[left] at (-12,-4){$\vdots$};
\draw [line width=2] (-12,-5) -- (-1,-5);
\node[left] at (-12,-5){\tiny{$C^{\bullet}(1) \tnsr \ExtAlg^gG^*$}};
\draw [line width=2] (-12,-6) -- (0,-6);
\node[left] at (-12,-6){\tiny{$C^{\bullet}(0) \tnsr \ExtAlg^gG^*$}};
\draw [line width=2] (-12,-7) -- (0,-7);
\node[left] at (-12,-7){\tiny{$C^{\bullet}(-1) \tnsr \ExtAlg^gG^*$}};
\node at (-6,-8) {$\vdots$};
\node at (-6,14) {$\vdots$};
\draw [dashed] (-12, 6) -- (0,-6); 
\draw [dashed] (0, -6) -- (0,-8);
\draw [->] (-11,6) -- (-7,6);
\draw [->] (-10,5) -- (-6,5);
\draw [->] (-9,4) -- (-5,4);
\draw [->] (-8,3) -- (-4,3);
\draw [->] (-7,2) -- (-3,2);
\draw [->] (-6,1) -- (-2,1);
\draw [->] (-5,0) -- (-1,0);
\node[below] at (-6,0) {\tiny{$(-g,0)$}};
\node[below] at (-6,6) {\tiny{$(-(f-g),f-g)$}};
\draw [dashed] (-14, 3) -- (2,3);
\end{tikzpicture}
\label{fig:D}
\caption{The complexes $D^{\bullet}(i)$}
\end{figure}
\FloatBarrier

There is a duality for the family of complexes $D^{\bullet}(i)$: up to an appropriate shift we have that 

\[
\Hom(D^{\bullet}(i), R) \niso D^{\bullet}(f-g-i) \tnsr \ExtAlg^f F^* \tnsr \ExtAlg^g G
\]

\subsection{Cohomology of $D^{\bullet}(i)$} Note that in $_{\infty}^{col}E_{\bullet}^{\bullet}$ contains at most one nonzero term on each diagonal and since the spectral sequence converges, 

\begin{equation}
\label{eq:tot2}
_{\infty}^{col}E^{p,q} \niso H^{p+q}(\tot p_* \scrptI^{\bullet, \bullet}) 
\end{equation}

Moreover, we can compute $_{\infty}^{col}E^{p,q}$ from the complexes $D^{\bullet}(i)$. So combining with equations~\ref{eq:tot} and~\ref{eq:tot2} we get 

\begin{equation}
\R^n p_* \scrptO_Z(i) \niso \mbox{cohomology of $D^{\bullet}(i)$ at the intersection with the line $p+q=n$}
\end{equation}

Let us draw a picture of the $\infty$-pages of both spectral sequences superimposed on each other. The expression above means that $D^{\bullet}(i)$ has a chance to "pick up" cohomology only at intersections with diagonals in the grey region. For example, in the picture below (which was generated using values $f=12$, $g=6$, $i=3$) we see that we only have one such intersection, namely at $(0,0)$, meaning that $D^{\bullet}(i)$ is a resolution in this case.

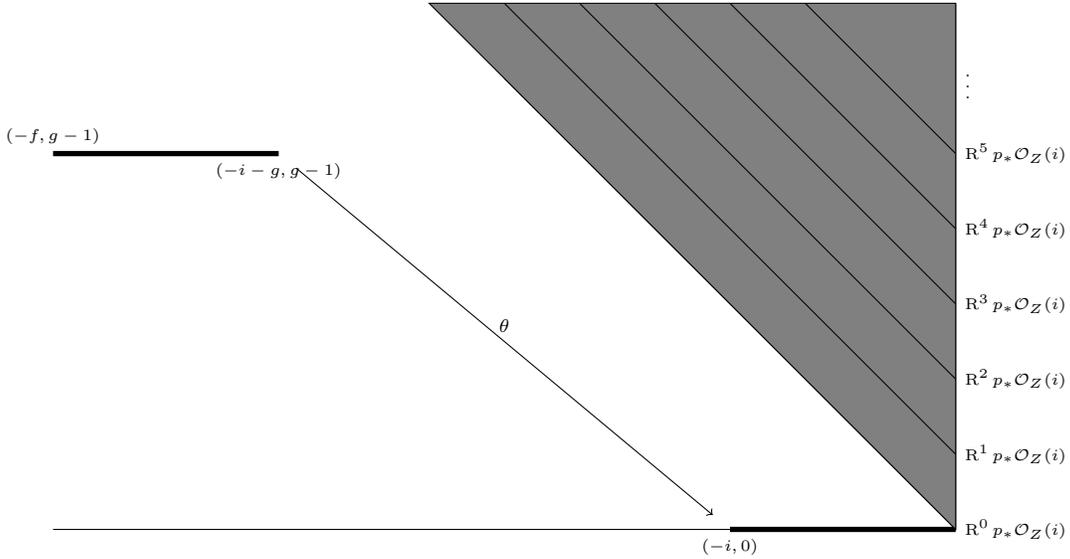
\begin{figure}[h!]
\begin{tikzpicture}[scale=1]
\draw (-12,0) -- (0,0);
\draw (0,0) -- (0, 7);
\draw (0,0) -- (-6,6);
\draw [line width=2] (-12, 5) -- (-9, 5);
\draw [line width=2] (-3,0) -- (0,0);
\node[above] at (-12,5){\tiny{$(-f,g-1)$}};
\node[below] at (-9,5){\tiny{$(-i-g,g-1)$}};
\node[below] at (-3,0){\tiny{$(-i,0)$}};
\node[right] at (0,0){\tiny{$\R^0 p_* \scrptO_Z(i)$}};
\node[right] at (0,1){\tiny{$\R^1 p_* \scrptO_Z(i)$}};
\node[right] at (0,2){\tiny{$\R^2 p_* \scrptO_Z(i)$}};
\node[right] at (0,3){\tiny{$\R^3 p_* \scrptO_Z(i)$}};
\node[right] at (0,4){\tiny{$\R^4 p_* \scrptO_Z(i)$}};
\node[right] at (0,5){\tiny{$\R^5 p_* \scrptO_Z(i)$}};
\node[right] at (0,6){\tiny{$\vdots$}};
\draw[fill=gray]  (0,0) -- (-7,7) -- (0,7) -- cycle;
\draw  (0,1) -- (-6,7);
\draw  (0,2) -- (-5,7);
\draw  (0,3) -- (-4,7);
\draw  (0,4) -- (-3,7);
\draw  (0,5) -- (-2,7);
\path[->,font=\scriptsize] 
(-9,5) edge[shorten <= .3cm, shorten >= .3cm]node[above]{$\theta$} (-3,0);
\end{tikzpicture}
\caption{$f=12$, $g=6$, $i=3$}
\end{figure}
\FloatBarrier

Let us see what happens when we vary $i$. In general, as $i$ decreases from being very positive to very negative, the lower part $K^{\bullet}(i)$ shrinks, but the upper part $C^{\bullet}(i+g)$ stretches. Only the upper part can intersect with the grey region, and this determines when and where the cohomology of $D^{\bullet}(i)$ has to be.

So, let us start decreasing $i$. When $i=1$ we see the Buchsbaum-Rim complex. The lower part only has two terms: $S_0 G \tnsr F \niso F \to S_1 G = G$.

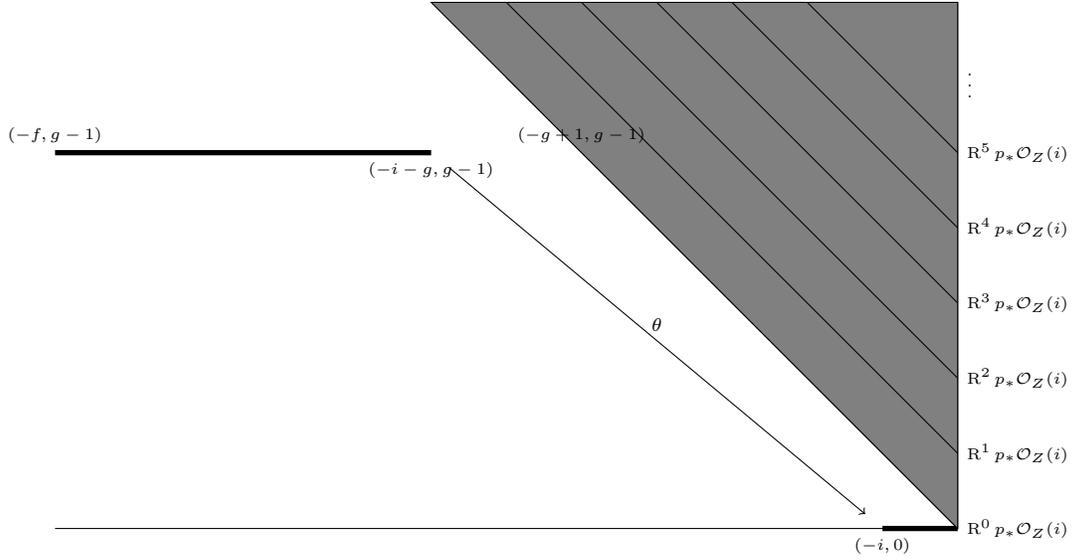
\begin{figure}[h!]
\begin{tikzpicture}[scale=1]
\draw (-12,0) -- (0,0);
\draw (0,0) -- (0, 7);
\draw (0,0) -- (-6,6);
\draw [line width=2] (-12, 5) -- (-7, 5);
\draw [line width=2] (-1,0) -- (0,0);
\node[above] at (-12,5){\tiny{$(-f,g-1)$}};
\node[below] at (-7,5){\tiny{$(-i-g,g-1)$}};
\node[below] at (-1,0){\tiny{$(-i,0)$}};
\node[right] at (0,0){\tiny{$\R^0 p_* \scrptO_Z(i)$}};
\node[right] at (0,1){\tiny{$\R^1 p_* \scrptO_Z(i)$}};
\node[right] at (0,2){\tiny{$\R^2 p_* \scrptO_Z(i)$}};
\node[right] at (0,3){\tiny{$\R^3 p_* \scrptO_Z(i)$}};
\node[right] at (0,4){\tiny{$\R^4 p_* \scrptO_Z(i)$}};
\node[right] at (0,5){\tiny{$\R^5 p_* \scrptO_Z(i)$}};
\node[right] at (0,6){\tiny{$\vdots$}};
\draw[fill=gray]  (0,0) -- (-7,7) -- (0,7) -- cycle;
\draw  (0,1) -- (-6,7);
\draw  (0,2) -- (-5,7);
\draw  (0,3) -- (-4,7);
\draw  (0,4) -- (-3,7);
\draw  (0,5) -- (-2,7);
\node[above] at (-5,5){\tiny{$(-g+1, g-1)$}};
\path[->,font=\scriptsize] 
(-7,5) edge[shorten <= .3cm, shorten >= .3cm]node[above]{$\theta$} (-1,0);
\end{tikzpicture}
\caption{$f=12$, $g=6$, $i=1$}
\end{figure}
\FloatBarrier

Decrease $i$ by one again: put $i=0$. At this point the lower part still exists, but only has one term in it, namely $S_0 G = R$. This is the Eagon-Northcott complex, which resolves $R/I_g$ - the quotient by the ideal of maximal minors of $\phi$

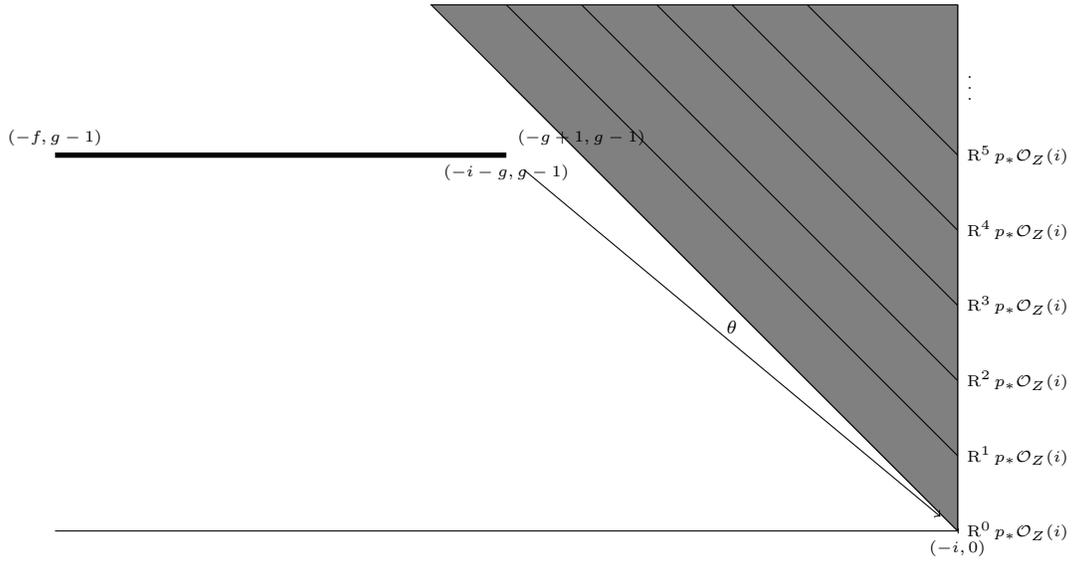
\begin{figure}[h!]
\begin{tikzpicture}[scale=1]
\draw (-12,0) -- (0,0);
\draw (0,0) -- (0, 7);
\draw (0,0) -- (-6,6);
\draw [line width=2] (-12, 5) -- (-6, 5);
\draw [line width=2] (0,0) -- (0,0);
\node[above] at (-12,5){\tiny{$(-f,g-1)$}};
\node[below] at (-6,5){\tiny{$(-i-g,g-1)$}};
\node[below] at (0,0){\tiny{$(-i,0)$}};
\node[right] at (0,0){\tiny{$\R^0 p_* \scrptO_Z(i)$}};
\node[right] at (0,1){\tiny{$\R^1 p_* \scrptO_Z(i)$}};
\node[right] at (0,2){\tiny{$\R^2 p_* \scrptO_Z(i)$}};
\node[right] at (0,3){\tiny{$\R^3 p_* \scrptO_Z(i)$}};
\node[right] at (0,4){\tiny{$\R^4 p_* \scrptO_Z(i)$}};
\node[right] at (0,5){\tiny{$\R^5 p_* \scrptO_Z(i)$}};
\node[right] at (0,6){\tiny{$\vdots$}};
\draw[fill=gray]  (0,0) -- (-7,7) -- (0,7) -- cycle;
\draw  (0,1) -- (-6,7);
\draw  (0,2) -- (-5,7);
\draw  (0,3) -- (-4,7);
\draw  (0,4) -- (-3,7);
\draw  (0,5) -- (-2,7);
\node[above] at (-5,5){\tiny{$(-g+1, g-1)$}};
\path[->,font=\scriptsize] 
(-6,5) edge[shorten <= .3cm, shorten >= .3cm]node[above]{$\theta$} (0,0);
\end{tikzpicture}
\caption{$f=12$, $g=6$, $i=0$}
\end{figure}
\FloatBarrier

Next case $i=-1$ is somewhat exciting: lower part ceases to exist (and so no more splicing) and for the first time the upper part intersects the grey region - this is still a resolution.

\begin{figure}[h!]
\begin{tikzpicture}[scale=1]
\draw (-12,0) -- (0,0);
\draw (0,0) -- (0, 7);
\draw (0,0) -- (-6,6);
\draw [line width=2] (-12, 5) -- (-5, 5);
\node[above] at (-12,5){\tiny{$(-f,g-1)$}};
\node[right] at (0,0){\tiny{$\R^0 p_* \scrptO_Z(i)$}};
\node[right] at (0,1){\tiny{$\R^1 p_* \scrptO_Z(i)$}};
\node[right] at (0,2){\tiny{$\R^2 p_* \scrptO_Z(i)$}};
\node[right] at (0,3){\tiny{$\R^3 p_* \scrptO_Z(i)$}};
\node[right] at (0,4){\tiny{$\R^4 p_* \scrptO_Z(i)$}};
\node[right] at (0,5){\tiny{$\R^5 p_* \scrptO_Z(i)$}};
\node[right] at (0,6){\tiny{$\vdots$}};
\draw[fill=gray]  (0,0) -- (-7,7) -- (0,7) -- cycle;
\draw  (0,1) -- (-6,7);
\draw  (0,2) -- (-5,7);
\draw  (0,3) -- (-4,7);
\draw  (0,4) -- (-3,7);
\draw  (0,5) -- (-2,7);
\node[below] at (-5,5){\tiny{$(-i-g,g-1)$}};
\node[above] at (-5,5){\tiny{$(-g+1, g-1)$}};
\end{tikzpicture}
\caption{$f=12$, $g=6$, $i=-1$}
\end{figure}
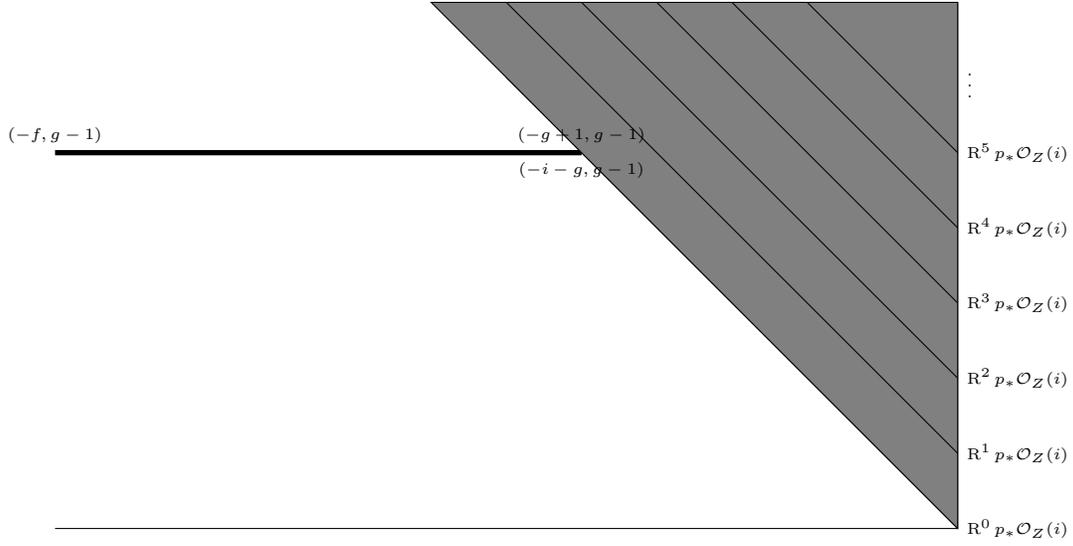
\FloatBarrier

When $i=-2$ there are two intersections, so $D^{\bullet}(-2)$ can have cohomology in two places. We will see (Theorem~\ref{thm:main}) that it actually happens.

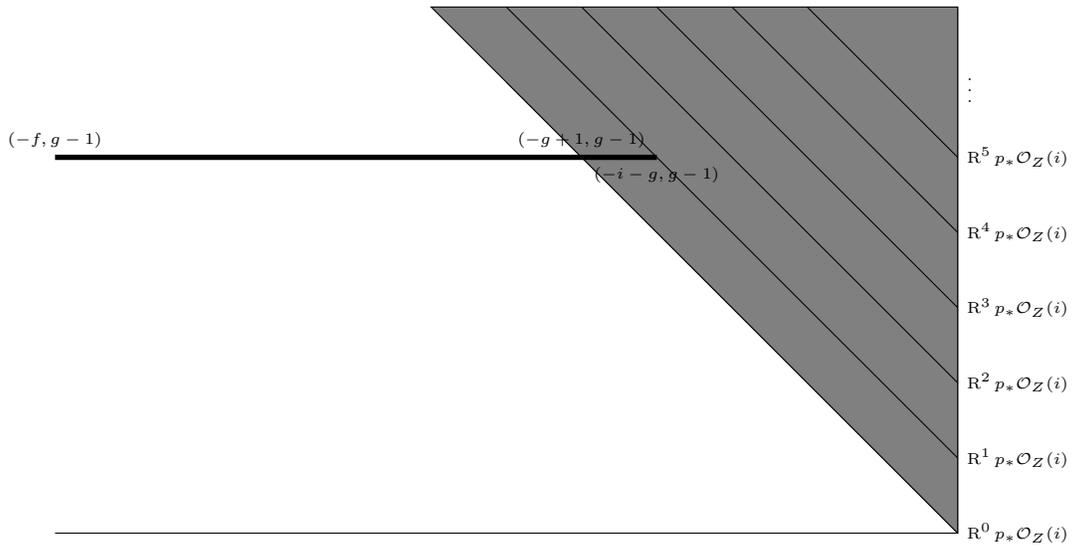
\begin{figure}[h!]
\begin{tikzpicture}[scale=1]
\draw (-12,0) -- (0,0);
\draw (0,0) -- (0, 7);
\draw (0,0) -- (-6,6);
\node[right] at (0,0){\tiny{$\R^0 p_* \scrptO_Z(i)$}};
\node[right] at (0,1){\tiny{$\R^1 p_* \scrptO_Z(i)$}};
\node[right] at (0,2){\tiny{$\R^2 p_* \scrptO_Z(i)$}};
\node[right] at (0,3){\tiny{$\R^3 p_* \scrptO_Z(i)$}};
\node[right] at (0,4){\tiny{$\R^4 p_* \scrptO_Z(i)$}};
\node[right] at (0,5){\tiny{$\R^5 p_* \scrptO_Z(i)$}};
\node[right] at (0,6){\tiny{$\vdots$}};
\draw[fill=gray]  (0,0) -- (-7,7) -- (0,7) -- cycle;
\draw  (0,1) -- (-6,7);
\draw  (0,2) -- (-5,7);
\draw  (0,3) -- (-4,7);
\draw  (0,4) -- (-3,7);
\draw  (0,5) -- (-2,7);
\draw [line width=2] (-12, 5) -- (-4, 5);
\node[above] at (-12,5){\tiny{$(-f,g-1)$}};
\node[below] at (-4,5){\tiny{$(-i-g,g-1)$}};
\node[above] at (-5,5){\tiny{$(-g+1, g-1)$}};
\end{tikzpicture}
\caption{$f=12$, $g=6$, $i=-2$}
\end{figure}
\FloatBarrier

Finally, let us look at $i=-6$. This is where all possible intersections happen for the first time. Decreasing $i$ further will not change the picture.

\begin{figure}[h!]
\begin{tikzpicture}[scale=1]
\draw (-12,0) -- (0,0);
\draw (0,0) -- (0, 7);
\draw (0,0) -- (-6,6);
\node[right] at (0,0){\tiny{$\R^0 p_* \scrptO_Z(i)$}};
\node[right] at (0,1){\tiny{$\R^1 p_* \scrptO_Z(i)$}};
\node[right] at (0,2){\tiny{$\R^2 p_* \scrptO_Z(i)$}};
\node[right] at (0,3){\tiny{$\R^3 p_* \scrptO_Z(i)$}};
\node[right] at (0,4){\tiny{$\R^4 p_* \scrptO_Z(i)$}};
\node[right] at (0,5){\tiny{$\R^5 p_* \scrptO_Z(i)$}};
\node[right] at (0,6){\tiny{$\vdots$}};
\draw[fill=gray]  (0,0) -- (-7,7) -- (0,7) -- cycle;
\draw  (0,1) -- (-6,7);
\draw  (0,2) -- (-5,7);
\draw  (0,3) -- (-4,7);
\draw  (0,4) -- (-3,7);
\draw  (0,5) -- (-2,7);
\draw [line width=2] (-12, 5) -- (0, 5);
\node[above] at (-12,5){\tiny{$(-f,g-1)$}};
\node[below] at (0,5){\tiny{$(-i-g,g-1)$}};
\node[above] at (-5,5){\tiny{$(-g+1, g-1)$}};
\end{tikzpicture}
\caption{$f=12$, $g=6$, $i=-6$}
\end{figure}
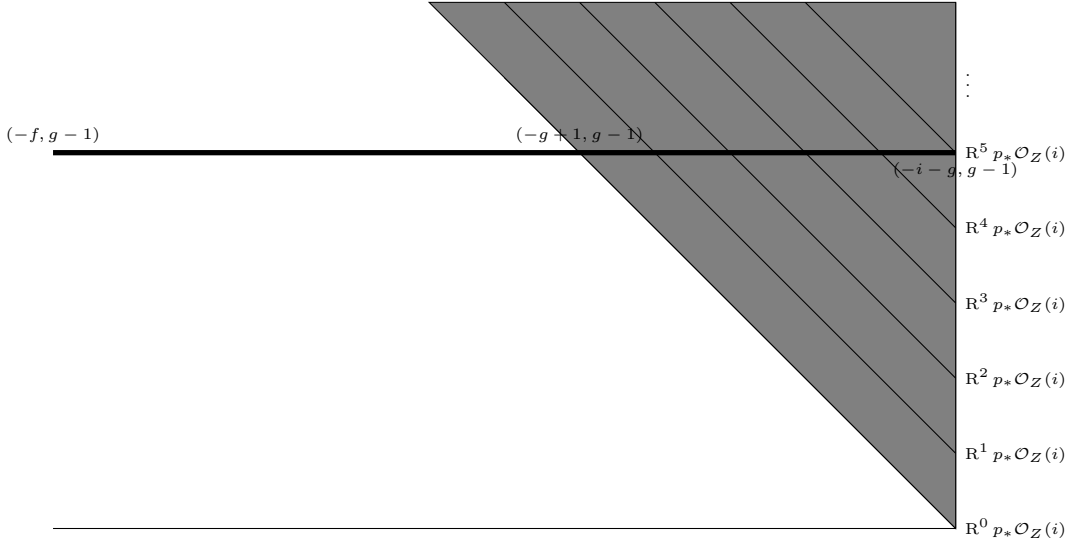
\FloatBarrier

With above pictures in mind, the proof of the following theorem should be clear. The statement about $\Ext$ follows from the duality $\Hom(D^{\bullet}(i), R) \niso D^{\bullet}(f-g-i) \tnsr \ExtAlg^f F^* \tnsr \ExtAlg^g G$.

\begin{theorem}
\label{thm:main}

(Kempf) Let $\phi \colon F \to G$ be the generic map. 

\begin{enumerate}
\item For $i \geq 0$ the complex $D^{\bullet}(i)$ is a resolution of $\R^0p_* \scrptO_Z(i) \niso \Sym^i \coker \phi$
\item For $i<0$ the complex $D^{\bullet}$ coincides with $C^{\bullet}(i+g)$ and all of it's cohomology is concentrated in the interval $[-g+1, min(0,-g-i)]$. Moreover, for each $j \in [-g+1, min(0,-g-i)]$ the cohomology $H^jD^{\bullet}(i)$ is non-zero and

\[
H^jD^{\bullet}(i) \niso \R^{j+g-1}p_* \scrptO_Z(i) \niso \Ext_R^{j+f}(\Sym_{f-g-i} \coker \phi, R) \tnsr \ExtAlg^fF \tnsr \ExtAlg^gG^*
\]

\end{enumerate}

\end{theorem}

The fact that for each $j \in [-g+1, min(0,-g-i)]$ the cohomology $H^jD^{\bullet}(i)$ is non-zero will be proven later in Section~\ref{sec:StrRep}. The Figure~\ref{fig:coh} shows the schematic picture of cohomology of complexes $D^{\bullet}(i)$. Note that by Theorem~\ref{thm:main} there is much more exactness than as stated in Theorem A2.10 in \cite{Eis}. Moreover, everything in Sections 2.16-2.18 of \cite{Br} follows also from Theorem~\ref{thm:main} (and its proof).

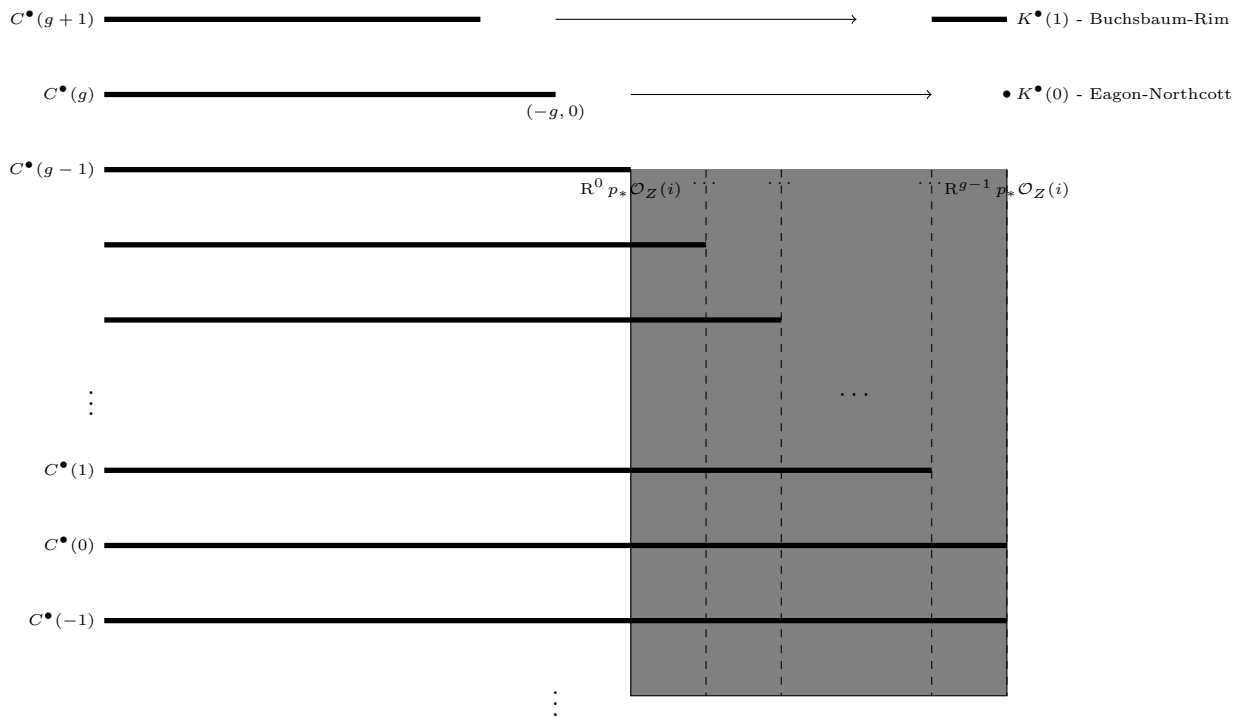
\begin{figure}
\begin{tikzpicture}[scale=1]
\draw[fill=gray]  (-5,-1) -- (-5,-8) -- (0,-8) -- (0,-1);
\draw [line width=2] (-1,1) -- (0,1);
\node[right] at (0,1){\tiny{$K^{\bullet}(1)$ - Buchsbaum-Rim}};
\node[right] at (0,0){\tiny{$K^{\bullet}(0)$ - Eagon-Northcott}};
\node[] at (0,0){\tiny{$\bullet$}};
\draw [line width=2] (-12,1) -- (-7,1);
\node[left] at (-12,1){\tiny{$C^{\bullet}(g+1)$}};
\draw [line width=2] (-12, 0) -- (-6,0);
\node[left] at (-12,0){\tiny{$C^{\bullet}(g)$}};
\draw [line width=2] (-12,-1) -- (-5,-1);
\node[left] at (-12,-1){\tiny{$C^{\bullet}(g-1)$}};
\draw [line width=2] (-12,-2) -- (-4,-2);
\draw [line width=2] (-12,-3) -- (-3,-3);
\node[left] at (-12,-4){$\vdots$};
\draw [line width=2] (-12,-5) -- (-1,-5);
\node[left] at (-12,-5){\tiny{$C^{\bullet}(1)$}};
\draw [line width=2] (-12,-6) -- (0,-6);
\node[left] at (-12,-6){\tiny{$C^{\bullet}(0)$}};
\draw [line width=2] (-12,-7) -- (0,-7);
\node[left] at (-12,-7){\tiny{$C^{\bullet}(-1)$}};
\node at (-6,-8) {$\vdots$};
\draw [->] (-6,1) -- (-2,1);
\draw [->] (-5,0) -- (-1,0);
\node[below] at (-6,0) {\tiny{$(-g,0)$}};
\node[below] at (-5,-1) {\tiny{$\R^0p_* \scrptO_Z(i)$}};
\draw [dashed] (-4,-1) -- (-4,-8);
\node[below] at (-4,-1) {\tiny{$\cdots$}};
\draw [dashed] (-3,-1) -- (-3,-8);
\node[below] at (-3,-1) {\tiny{$\cdots$}};
\draw [dashed] (-1,-1) -- (-1,-8);
\node[below] at (-1,-1) {\tiny{$\cdots$}};
\node[below] at (0,-1) {\tiny{$\R^{g-1}p_* \scrptO_Z(i)$}};
\draw [dashed] (0,-1) -- (0,-8);
\node at (-2, -4) {$\cdots$};
\end{tikzpicture}
\caption{The cohomology of $D^{\bullet}(i)$}
\label{fig:coh}
\end{figure}
\FloatBarrier

\begin{remark}
In fact, the results in Theorem~\ref{thm:main} remain valid for any map $\phi$ such that the associated Koszul complex of sheaves $\scrptK^{\bullet}(\phi)$ constructed as in Subsection~\ref{subsec:GenConst} is exact.
\end{remark}

\section{Structure of Cohomology of $D^{\bullet}(i)$ as a graded $H$-representation}
\label{sec:StrRep}

The Theorem~\ref{thm:main} of previous section reduces computation of cohomology of $D^{\bullet}(i)$ to computing the sheaf cohomology $R^{\bullet}p_* \scrptO_Z(i)$, so in this section we concentrate on this. Throughout this section we assume $i<0$.

The cohomology $\R^qp_* \scrptO_Z(i)$ has the structure of $R-$ (actually $R/I_g-$) module. First, we will calculate this cohomology group as a graded $H$-representation. To do that, we want to calculate $\pi'_* \R^q p_* (\iota_* \scrptO_Z(i))$. But from the commutativity of lower square in Figure~\ref{fig:setup} we can go in counter-clockwise direction, i.e. we can calculate $\R^q p'_* (\pi_* (\iota_*\scrptO_Z(i)))$ instead. We already know that $\pi_* (\iota_*\scrptO_Z(i)) \niso  \Sym (\scrptT(-1) \tnsr p'^*V)(i)$. Also, the Cauchy formula (see Section 6 of \cite{Fult-Harr}) gives us the isomorphism:

\[
\Sym_d (\scrptT(-1) \tnsr p'^*V) \niso \underset{|\lambda|=d}{\Drsum} \bbS_{\lambda} (\scrptT(-1)) \tnsr p'^* (\bbS_{\lambda} V)
\]

Using the projection formula, we see that in general the structure of $\R^q p_* (\iota_* \scrptO_Z(i))$ as a graded $H$-representation is described by the following equation:

\begin{equation}
\label{eq:vect}
\pi'_* \R^qp_* (\iota_*\scrptO_Z(i)) \niso \underset{\lambda}{\Drsum} (\R^q p'_* [\bbS_{\lambda} (\scrptT(-1))(i)]  \tnsr \bbS_{\lambda}V)
\end{equation}

The term $R^q p'_* [\bbS_{\lambda} (\scrptT(-1))(i)]$ was calculated in the Example~\ref{ex:coh}. In particular,

\begin{proposition}
\label{prop:coh1}
The cohomology $\R^qp_* \scrptO_Z(i)$ is non-zero if and only if 

\[
0 \leq q \leq min(-i-1,g-1)
\]

Moreover, $\R^qp_* \scrptO_Z(i)$ is finite-dimensional over $\Bbbk$ if and only if $q=g-1$ and $\R^{g-1}p_* \scrptO_Z(-g) \iso \Bbbk$.
\end{proposition}

\begin{example}
\label{example:1}

Consider the case when $\dim W^* = 2$, so our vector bundle is over $\bbP(W^*)=\bbP^1$. In this case $\scrptT(-1) \niso \scrptO_{\bbP^1}(1)$ is of rank $1$ and the Cauchy formula only gives one term in each degree:

\begin{eqnarray*}
(\Sym_d(\scrptO_{\bbP^1}(1) \tnsr p'^*V))(i) \niso \Sym_d (\scrptO_{\bbP^1}(1))(i) \tnsr p'^* \Sym_d V \niso \scrptO_{\bbP^1}(d+i) \tnsr p'^* \Sym_d V
\end{eqnarray*}

Now using the formula for the cohomology of line bundles on projective space we get:

\begin{eqnarray*}
\pi'_* \R^qp_* \scrptO_Z(i) \niso \underset{d \geq 0}{\Drsum} \R^qp'_* (\scrptO_{\bbP^1}(d+i) \tnsr p'^* \Sym_d V) \niso
 \\
\begin{cases}
\underset{d \geq 0}{\Drsum} \Sym_{d+i} W \tnsr \Sym_d V \text{  if  $q=0$}\\
\underset{d \geq 0}{\Drsum} \Sym^*_{-(d+i)-2} W \tnsr \ExtAlg^2 W^* \tnsr \Sym_d V \text{  if $q=1$}\\
\end{cases}
\end{eqnarray*}

\end{example}

\section{Category of $H$-Equivariant Modules} 
\label{sec:Equiv}

Recall that $H:=\GL(V) \times \GL(W^*)$, $R:=\Sym (V \tnsr W^*)$. Then $H$ acts on $R$ by ring automorphisms in the usual way. Let $A:=R \rtimes H$ denote the twisted group algebra. The multiplication is given by the rule $(r_1, g_1)(r_2, g_2)=(r_1g_1(r_2), g_1g_2)$. The grading on $A$ comes from the usual grading of $R$.

Let $N \in A-\fmod$ be a graded left $A$-module all of whose components are finite-dimensional vector spaces, then it is also a graded $R$-module via the inclusion $R \into A$ and a representation of $H$, moreover the multiplication map $\mu \colon R \tnsr N \to N$ is a map of representations. Conversely, a representation $N$ of $H$ with graded $R-$module structure such that the multiplication map $R \tnsr N \to M$ is a map of representations defines a graded module over the twisted group algebra.

\begin{definition}
We will call such a finitely generated module over $A$ an equivariant module.
\end{definition}

\begin{example}
\label{ex:multR}

Consider $R=\Sym(V \tnsr W^*) \niso \underset{\lambda}{\Drsum}( \bbS_{\lambda} V \tnsr \bbS_{\lambda} W^*)$. The module structure is uniquely determined by the action on $V \tnsr W^*$, i.e by the maps $(V\tnsr W^*) \tnsr (\bbS_{\lambda} V \tnsr \bbS_{\lambda} W^*) \to R$. Clearly this action is a map of representations. Pieri rules say

\[
(V \tnsr W^*) \tnsr (\bbS_{\lambda} V \tnsr \bbS_{\lambda} W^*) \niso \underset{\mu, \eta \in VS(\lambda,1)}{\Drsum} \bbS_{\mu} V \tnsr \bbS_{\eta} W^*
\]

but $R$ does not have any copies of $\bbS_{\mu} V \tnsr \bbS_{\eta} W^*$ unless $\mu=\eta$. Moreover, when there is a possibility for this map to be nonzero on some component (i.e on the component with $\mu=\eta$) it has to be nonzero (image has to be the ideal of elements of positive degree). Thus, the result of applying the action of $(V\tnsr W^*)$ to $\bbS_{\lambda} V \tnsr \bbS_{\lambda} W^*$ is $\underset{\mu \in VS(\lambda,1)}{\Drsum} \bbS_{\mu} V \tnsr \bbS_{\mu} W^*$. 
\end{example}

We will display the structure of an equivariant module $N$ by drawing a lattice. The vertices of the lattice are the representations that occur in $N$. An arrow going from a representation $\bbS_{\alpha}V \tnsr \bbS_{\beta} W^*$ to a representation $\bbS_{\mu}V \tnsr \bbS_{\eta} W^*$ means that $\bbS_{\mu}V \tnsr \bbS_{\eta} W^*$ lies in the image of the map $(V \tnsr W^*) \tnsr \bbS_{\alpha}V \tnsr \bbS_{\beta} W^* \to (V \tnsr W^*) \tnsr N \to N$. where the first map is the inclusion $\bbS_{\alpha}V \tnsr \bbS_{\beta} W^* \to \tnsr N$ tensored with $V \tnsr W^*$ and the second map is the multiplication in $N$.

\begin{example}

Consider the equivariant module $R$. It's lattice is the following:

\begin{figure}[h!]

\begin{tikzpicture}[scale=1.4]

\node at (0,0) {$\emptyset$};
\node at (0,1) {$\scriptsize\young(\hfil) \tnsr \scriptsize\young(\hfil)$};
\node at (1,2) {$\scriptsize\young(\hfil,\hfil) \tnsr \scriptsize\young(\hfil,\hfil)$};
\node at (-1,2) {$\scriptsize\young(\hfil \hfil) \tnsr \scriptsize\young(\hfil \hfil)$};
\node at (-2,3) {$\scriptsize\young(\hfil \hfil \hfil) \tnsr \scriptsize\young(\hfil \hfil \hfil)$};
\node at (0,3) {$\scriptsize\young(\hfil \hfil,\hfil) \tnsr \scriptsize\young(\hfil \hfil,\hfil)$};
\node at (2,3) {$\scriptsize\young(\hfil,\hfil,\hfil) \tnsr \scriptsize\young(\hfil,\hfil,\hfil)$};
\node at (-2,4) {$\vdots$};
\node at (0,4) {$\vdots$};
\node at (2,4) {$\vdots$};

\path[->,font=\scriptsize] 

(0,0) edge[shorten <= .5cm, shorten >= .5cm] (0,1)
(0,1) edge[shorten <= .5cm, shorten >= .5cm] (-1,2) 
(0,1) edge[shorten <= .5cm, shorten >= .5cm] (1,2) 
(-1,2) edge[shorten <= .5cm, shorten >= .5cm] (-2,3) 
(-1,2) edge[shorten <= .5cm, shorten >= .5cm] (0,3) 
(1,2) edge[shorten <= .5cm, shorten >= .5cm] (0,3) 
(1,2) edge[shorten <= .5cm, shorten >= .5cm] (2,3) 
(-2,3) edge[shorten <= .5cm, shorten >= .5cm] (-2,4) 
(0,3) edge[shorten <= .5cm, shorten >= .5cm] (0,4) 
(2,3) edge[shorten <= .5cm, shorten >= .5cm] (2,4); 

\end{tikzpicture}

\caption{Lattice of $R$ as an $H$-equivariant module}

\end{figure}
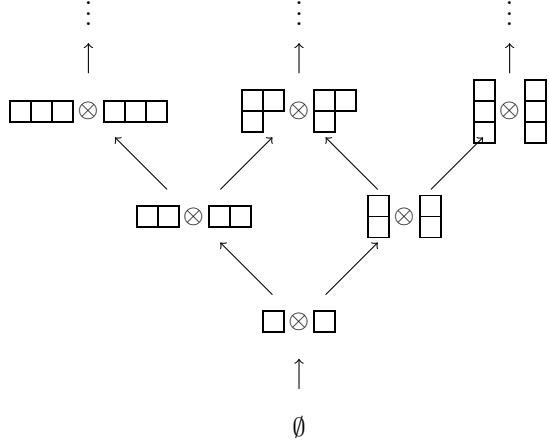
\FloatBarrier

Here the representation on the left of the tensor product sign is applied to $V$ and the representation of the right of the tensor product sign is applied to $W^*$. For example, $\scriptsize\young(\hfil \hfil,\hfil) \tnsr \scriptsize\young(\hfil \hfil,\hfil)$ means $\bbS_{(2,1,0 \cdots)} V \tnsr \bbS_{(2,1,0, \cdots)} W^*$.
 
\end{example}

\begin{example}
\label{ex:ideal}

For any partition $\lambda$ we have a unique $H$-equivariant ideal $I_{\lambda}$ in $R$ generated by $\bbS_{\lambda}V \tnsr \bbS_{\lambda} W^*$. The ideal $I_{\lambda}$ decomposes into representations: $I_{\lambda} \niso \underset{\mu \supseteq \lambda}{\Drsum} \bbS_{\mu} V \tnsr \bbS_{\mu} W^*$. The structure of equivariant ideals of $R$ was also described in \cite{de-Conc-Eis-Proc}. For example, if $\lambda=(2,2,0, \cdots)$, then the lattice of $I_{\lambda}$ looks as follows:

\begin{figure}[h!]

\begin{tikzpicture}[scale=1.4]

\node at (0,0) {$\scriptsize\young(\hfil \hfil,\hfil \hfil) \tnsr \scriptsize\young(\hfil \hfil,\hfil \hfil)$};
\node at (-1,1) {$\scriptsize\young(\hfil \hfil \hfil,\hfil \hfil) \tnsr \scriptsize\young(\hfil \hfil \hfil,\hfil \hfil)$};
\node at (1,1) {$\scriptsize\young(\hfil \hfil,\hfil \hfil,\hfil) \tnsr \scriptsize\young(\hfil \hfil,\hfil \hfil,\hfil)$};
\node at (-4,2) {$\scriptsize\young(\hfil \hfil \hfil \hfil,\hfil \hfil) \tnsr \scriptsize\young(\hfil \hfil \hfil \hfil,\hfil \hfil)$};
\node at (-2,2) {$\scriptsize\young(\hfil \hfil \hfil,\hfil \hfil \hfil) \tnsr \scriptsize\young(\hfil \hfil \hfil,\hfil \hfil \hfil)$};
\node at (0,2) {$\scriptsize\young(\hfil \hfil \hfil,\hfil \hfil,\hfil) \tnsr \scriptsize\young(\hfil \hfil \hfil,\hfil \hfil,\hfil)$};
\node at (2,2) {$\scriptsize\young(\hfil \hfil,\hfil \hfil,\hfil \hfil) \tnsr \scriptsize\young(\hfil \hfil,\hfil \hfil,\hfil \hfil)$};
\node at (4,2) {$\scriptsize\young(\hfil \hfil,\hfil \hfil,\hfil,\hfil) \tnsr \scriptsize\young(\hfil \hfil,\hfil \hfil,\hfil,\hfil)$};

\node at (-4,3) {$\vdots$};
\node at (-2,3) {$\vdots$};
\node at (0,3) {$\vdots$};
\node at (2,3) {$\vdots$};
\node at (4,3) {$\vdots$};

\path[->,font=\scriptsize] 

(0,0) edge[shorten <= .5cm, shorten >= .5cm] (-1,1)
(0,0) edge[shorten <= .5cm, shorten >= .5cm] (1,1) 
(-1,1) edge[shorten <= .5cm, shorten >= .5cm] (-4,2) 
(-1,1) edge[shorten <= .5cm, shorten >= .5cm] (-2,2) 
(-1,1) edge[shorten <= .5cm, shorten >= .5cm] (0,2) 
(1,1) edge[shorten <= .5cm, shorten >= .5cm] (0,2) 
(1,1) edge[shorten <= .5cm, shorten >= .5cm] (2,2) 
(1,1) edge[shorten <= .5cm, shorten >= .5cm] (4,2) 
(-4,2) edge[shorten <= .5cm, shorten >= .5cm] (-4,3) 
(-2,2) edge[shorten <= .5cm, shorten >= .5cm] (-2,3) 
(0,2) edge[shorten <= .5cm, shorten >= .5cm] (0,3) 
(2,2) edge[shorten <= .5cm, shorten >= .5cm] (2,3) 
(4,2) edge[shorten <= .5cm, shorten >= .5cm] (4,3); 

\end{tikzpicture}

\caption{Lattice of $I_{(2,2,0,\cdots)}$ as an $H$-equivariant module}

\end{figure}
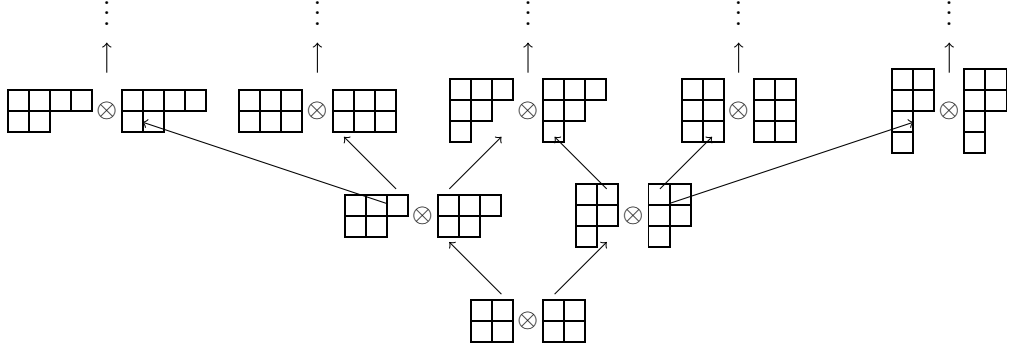
\FloatBarrier

\end{example}

As another example, consider the following question: given some representation $\bbS_{\mu}V \tnsr \bbS_{\eta}W^* \in R \tnsr (\ExtAlg^k V \tnsr \Sym_l W^*)$ what does this representation generate in the next degree? We will call $\bbS_{\mu}V$ the $\GL(V)$-part and $\bbS_{\eta}W^*$ the $\GL(W^*)$-part of $\bbS_{\mu}V \tnsr \bbS_{\eta}W^* \tnsr \ExtAlg^g W^*$ respectively. 

Each representation $\bbS_{\mu}V \tnsr \bbS_{\eta}W^*$ in $R \tnsr (\ExtAlg^k V \tnsr \Sym_l W^*)$ is formed by first choosing some partition $\tau$ and then adding extra $k$ boxes according to the Pieri rule for exterior powers to the diagram of $\tau$ to form the $\GL(V)$-part. We will mark these added boxes with "$\ExtAlg$'s. The $\GL(W^*)$-part is formed by adding extra $l$ boxes to the diagram of $\tau$ according to the Pieri rule for the symmetric powers. We will mark these added boxes with "$S$"'s. Such labeled diagram of $\bbS_{\mu}V \tnsr \bbS_{\eta}W^*$ records the embedding of the representation into $\bbS_{\tau}V \tnsr \bbS_{\tau} W^* \tnsr (\ExtAlg^k V \tnsr \Sym_l W^*)$.

Given such labeled diagram $\bbS_{\mu}V \tnsr \bbS_{\eta}W^*$ we want to know the image of the map

\[
(V \tnsr W^*) \tnsr (\bbS_{\mu}V \tnsr \bbS_{\eta}W^*) \to (V \tnsr W^*) \tnsr (\bbS_{\tau}V \tnsr \bbS_{\tau} W^*) \tnsr (\ExtAlg^k V \tnsr \Sym_l W^*) \to R \tnsr (\ExtAlg^k V \tnsr \Sym_l W^*)
\]

where the first map is the embedding according to the given labeling tensored with $V \tnsr W^*$ and the second map is the multiplication in $R$. All the representations $\bbS_{\alpha}V \tnsr \bbS_{\beta} W^*$ in $(V \tnsr W^*) \tnsr (\bbS_{\mu}V \tnsr \bbS_{\eta}W^*)$ are formed by adding one more box to the $\GL(V)$-part which we will label with "$V$" and to $\GL(W^*)$ which we will label with "$W$". Thus for each irreducible summand $\bbS_{\alpha}V \tnsr \bbS_{\beta} W^*$ of $(V \tnsr W^*) \tnsr (\bbS_{\mu}V \tnsr \bbS_{\eta}W^*)$, its embedding into $(V \tnsr W^*) \tnsr \bbS_{\tau}V \tnsr \bbS_{\tau} W^* \tnsr (\ExtAlg^k V \tnsr \Sym_l W^*)$ is given to us by taking the embedding $\bbS_{\alpha} V \to V \tnsr (\bbS_{\tau}V \tnsr \ExtAlg^k V)$ which is given by the labeled diagram of $\alpha$ and tensoring it with the embedding $\bbS_{\beta} W^* \to W^* \tnsr (\bbS_{\tau}W^* \tnsr \Sym_l W^*)$  given by labeled diagram of $\beta$. In order to perform the multiplication in $R$ we need to know the labeled diagrams of $\alpha$ and $\beta$ according to the bracket placements $(V \tnsr \bbS_{\tau}V) \tnsr \ExtAlg^k V$ and $(W^* \tnsr \bbS_{\tau} W^*) \tnsr \Sym_l W^*)$, i.e. we need to "push "$V$" and "$W$" inside as in subsection~\ref{subsec:Emb}. Then the image of $\bbS_{\alpha}V \tnsr \bbS_{\tau} W^*$ will be $\bbS_{\alpha}V \tnsr \bbS_{\tau} W^*$ if "$V$" and "$W$" end up attached to the same places to the diagram of $\tau$ and zero otherwise (see Example~\ref{ex:multR}).

\begin{example}
Let $\tau=\scriptsize\young(\hfil \hfil,\hfil \hfil,\hfil \hfil)$. Let $\bbS_{\mu} V \tnsr \bbS_{\eta} W^*=\scriptsize\young(\hfil \hfil \ExtAlg,\hfil \hfil \ExtAlg,\hfil \hfil) \tnsr \scriptsize\young(\hfil \hfil,\hfil \hfil,\hfil \hfil,SS)$. All the summands of $(V \tnsr W^*) \tnsr \bbS_{\mu} V \tnsr \bbS_{\eta} W^*$ are:

$\scriptsize\young(\hfil \hfil \ExtAlg V,\hfil \hfil \ExtAlg,\hfil \hfil) \tnsr \scriptsize\young(\hfil \hfil W,\hfil \hfil,\hfil \hfil,SS)$, $\scriptsize\young(\hfil \hfil \ExtAlg V,\hfil \hfil \ExtAlg,\hfil \hfil) \tnsr \scriptsize\young(\hfil \hfil,\hfil \hfil,\hfil \hfil,SS,W)$, \\

$\scriptsize\young(\hfil \hfil \ExtAlg,\hfil \hfil \ExtAlg,\hfil \hfil V) \tnsr \scriptsize\young(\hfil \hfil W,\hfil \hfil,\hfil \hfil,SS)$, $\scriptsize\young(\hfil \hfil \ExtAlg,\hfil \hfil \ExtAlg,\hfil \hfil V) \tnsr\scriptsize\young(\hfil \hfil,\hfil \hfil,\hfil \hfil,SS,W)$,\\

$\scriptsize\young(\hfil \hfil \ExtAlg,\hfil \hfil \ExtAlg,\hfil \hfil,V) \tnsr \scriptsize\young(\hfil \hfil W,\hfil \hfil,\hfil \hfil,SS)$, $\scriptsize\young(\hfil \hfil \ExtAlg,\hfil \hfil \ExtAlg,\hfil \hfil,V) \tnsr \scriptsize\young(\hfil \hfil,\hfil \hfil,\hfil \hfil,SS,W)$.\\
 
Now, as in subsection~\ref{subsec:Emb}, we "push "$V$" and "$W$" inside":

$\scriptsize\young(\hfil \hfil V\ExtAlg,\hfil \hfil \ExtAlg,\hfil \hfil) \tnsr \scriptsize\young(\hfil \hfil W,\hfil \hfil,\hfil \hfil,SS)$, $\scriptsize\young(\hfil \hfil V\ExtAlg,\hfil \hfil \ExtAlg,\hfil \hfil) \tnsr \scriptsize\young(\hfil \hfil,\hfil \hfil,\hfil \hfil,WS,S)$,\\

$\scriptsize\young(\hfil \hfil V,\hfil \hfil \ExtAlg,\hfil \hfil \ExtAlg) \tnsr \scriptsize\young(\hfil \hfil W,\hfil \hfil,\hfil \hfil,SS)$, $\scriptsize\young(\hfil \hfil V,\hfil \hfil \ExtAlg,\hfil \hfil \ExtAlg) \tnsr\scriptsize\young(\hfil \hfil,\hfil \hfil,\hfil \hfil,WS,S)$,\\

$\scriptsize\young(\hfil \hfil \ExtAlg,\hfil \hfil \ExtAlg,\hfil \hfil,V) \tnsr \scriptsize\young(\hfil \hfil W,\hfil \hfil,\hfil \hfil,SS)$, $\scriptsize\young(\hfil \hfil \ExtAlg,\hfil \hfil \ExtAlg,\hfil \hfil,V) \tnsr \scriptsize\young(\hfil \hfil,\hfil \hfil,\hfil \hfil,WS,S)$.\\

The multiplication map in $R$ only selects the following:
$\scriptsize\young(\hfil \hfil V\ExtAlg,\hfil \hfil \ExtAlg,\hfil \hfil) \tnsr \scriptsize\young(\hfil \hfil W,\hfil \hfil,\hfil \hfil,SS)$, $\scriptsize\young(\hfil \hfil V,\hfil \hfil \ExtAlg,\hfil \hfil \ExtAlg) \tnsr \scriptsize\young(\hfil \hfil W,\hfil \hfil,\hfil \hfil,SS)$ and $\scriptsize\young(\hfil \hfil \ExtAlg,\hfil \hfil \ExtAlg,\hfil \hfil,V) \tnsr \scriptsize\young(\hfil \hfil,\hfil \hfil,\hfil \hfil,WS,S)$. Thus the representation $\scriptsize\young(\hfil \hfil \ExtAlg,\hfil \hfil \ExtAlg,\hfil \hfil) \tnsr \scriptsize\young(\hfil \hfil,\hfil \hfil,\hfil \hfil,SS)$ generates the representations $\scriptsize\young(\hfil \hfil \hfil \ExtAlg,\hfil \hfil \ExtAlg,\hfil \hfil) \tnsr \scriptsize\young(\hfil \hfil \hfil,\hfil \hfil,\hfil \hfil,SS)$, $\scriptsize\young(\hfil \hfil \hfil,\hfil \hfil \ExtAlg,\hfil \hfil \ExtAlg) \tnsr \scriptsize\young(\hfil \hfil \hfil,\hfil \hfil,\hfil \hfil,SS)$ and $\scriptsize\young(\hfil \hfil \ExtAlg,\hfil \hfil \ExtAlg,\hfil \hfil,\hfil) \tnsr \scriptsize\young(\hfil \hfil,\hfil \hfil,\hfil \hfil,\hfil S,S)$ in the next degree. Here symbols "$\ExtAlg$"'s and "$S$"'s tell us the embeddings of these representations into $R \tnsr (\ExtAlg^k V \tnsr \Sym_l W^*)$. So now one can \emph{in principle} continue to describe the entire submodule generated by the given representation.

\end{example}

\begin{remark}
\label{rem}
Note that when we "push the "$V$" inside" on the $\GL(V)$-part of the representation all the "$\ExtAlg$"'s remain in the same rows unless "$V$" is adjoined directly below "$\ExtAlg$"'s and there is no "$\ExtAlg$" to the left of the "$V$". For example, for the labeled diagram $\scriptsize\young(\hfil \hfil \ExtAlg,\hfil \hfil \ExtAlg,\hfil \hfil V)$ one has to "slide up" the "$V$" to get $\scriptsize\young(\hfil \hfil V,\hfil \hfil \ExtAlg,\hfil \hfil \ExtAlg)$, so "$\ExtAlg$"'s end up in different rows.  
\end{remark}

\section{Structure of Cohomology of $D^{\bullet}(i)$ as a graded $H$-Equivariant Module}
\label{sec:EquivStr}

Fix $i,q,f$ and $g$. In this section we will use the abbreviations $k:=-q+g-1$ and $l:=-i-q-1$. We will often ignore tensoring by $\ExtAlg^gW^*$, because it does not really matter. In this section we investigate the structure of cohomology of complexes $D^{\bullet}(i)$ as an equivariant $H$-module.  

As we saw in Section~\ref{sec:StrRep} the cohomology $\R^qp_* \scrptO_Z(i)$ as a graded $H$-representation is 

\begin{equation}
\label{eq:StrRep}
\R^qp_* \scrptO_Z(i) \niso \underset{\lambda \in A(k,l)}{\Drsum} \bbS_{\lambda}V \tnsr \bbS_{\delta(\lambda)} W^*
\end{equation}

where $A(k,l)$ is the set of all partitions $\lambda=(\lambda_1, \cdots \lambda_{g-1})$ satisfying the conditions $\lambda_{k+1}<l+1 \leq \lambda_k$ and $\delta(\lambda)$ is as in Example~\ref{ex:coh}.

Recall that $\R^qp_* \scrptO_Z(i) \niso H^{q-g+1}D^{\bullet}(i)=H^{-k}D^{\bullet}(i)$, so it has to be a subquotient of $\ExtAlg^k F \tnsr \Sym^*_l G \tnsr \ExtAlg^g G^*$. We have the following decomposition:

\begin{equation}
\begin{aligned}
\label{eq:decomp}
\ExtAlg^k F \tnsr \Sym^*_l G \tnsr \ExtAlg^g G^* \niso R \tnsr (\ExtAlg^k V \tnsr \Sym^*_l W \tnsr \ExtAlg^g W^*) \niso \\
\underset{\tau}{\Drsum} \underset{\substack{\mu \in VS(\tau, k), \\ \eta \in HS(\tau, l)}}{\Drsum} \bbS_{\mu}V \tnsr \bbS_{\eta} W^* \tnsr \ExtAlg^g W^* 
\end{aligned}
\end{equation}

Let $N \subset R \tnsr (\ExtAlg^k V \tnsr \Sym_l W^*)$ denote the submodule of $R \tnsr (\ExtAlg^k V \tnsr \Sym_l W^*)$ generated by all the representations $\bbS_{\mu} V \tnsr \bbS_{\eta} W^*$ such that the length of $(k+1)$-st row of $\mu$ is greater or equal to $l+1$. Then clearly $\R^qp_* \scrptO_Z(i)$ has to be a subquotient of $R \tnsr (\ExtAlg^k V \tnsr \Sym_l W^*)/N$.

Consider the representation

\[
\rho:=\bbS_{(l+1)^k}V \tnsr \bbS_{(l)^{k+1}}W^*
\]

which occurs only once in $R \tnsr (\ExtAlg^k V \tnsr \Sym_l W^*)$. Namely, when in the double sum~\ref{eq:decomp} we take $\tau=(l)^k$, $\mu$ is obtained from $\tau$ by adding all $k$ boxes to the $(l+1)$-st column and $\eta$ is obtained by adding all $l$ boxes to the $(k+1)$-st row. This summand sits in 

\[
(\bbS_{(l)^k}V \tnsr \bbS_{(l)^k} W^*) \tnsr (\ExtAlg^k V \tnsr \Sym_l W^* \tnsr \ExtAlg^g W^*)
\]

For example, when $k=l=2$, $\rho=\scriptsize\young(\hfil \hfil \ExtAlg,\hfil \hfil \ExtAlg) \tnsr \scriptsize\young(\hfil \hfil,\hfil \hfil,SS)$. Let us consider the submodule $<\rho>$ generated by $\rho$ in $R \tnsr (\ExtAlg^k V \tnsr \Sym_l W^*)$. 
For concriteness (however we will not use this) let us display the generators of $<\rho>$. The image of $\rho$ in $D^{-k}(i)=R \tnsr (\ExtAlg^k V \tnsr \Sym^*_l W \tnsr \ExtAlg^g W^*)$ is the image of the following (injective) composition involving the Pieri inclusions (see Section~\ref{subsec:Pieri}) and the Cauchy formula (see Section~\ref{subsec:RepH}):

\begin{eqnarray*}
\bbS_{(l+1)^k} V \tnsr \bbS_{(l)^{k+1}}W^* \into \\
\ExtAlg^{(k)^{l+1}} V \tnsr \ExtAlg^{(k+1)^{l}} W^* \to \ExtAlg^k V \tnsr \ExtAlg^{(k)^{l}}V \tnsr (\ExtAlg^k W^* \tnsr W^*)^{\tnsr l} \to \\
(\ExtAlg^k V \tnsr \ExtAlg^k W^*)^{\tnsr l} \tnsr \ExtAlg^k V \tnsr (W^*)^{\tnsr l} \to \\ 
(\Sym_k (V \tnsr W^*))^{\tnsr l} \tnsr \ExtAlg^k V \tnsr \Sym_l W^* \to \\
\Sym_{(kl)}(V \tnsr W^*) \tnsr \ExtAlg^k V \tnsr \Sym_l W^*
\end{eqnarray*}

A basis for $\bbS_{(l+1)^k} V \tnsr \bbS_{(l)^{k+1}}W^*$ consists of the ordered pairs of standard tableaux $(T,S)$ (one for $V$, one for $W^*$) of shapes $(l+1)^k$ and $(l)^{k+1}$ respectively. The image of the basis element corresponding to such a pair of tableaux in $\ExtAlg^{(k)^{l+1}} V \tnsr \ExtAlg^{(k+1)^{l}} W^*$ is 
\[
(e_{T(1,1)} \wdg \cdots e_{T(k,1)}) \tnsr \cdots (e_{T(1,l+1)} \wdg \cdots e_{T(k,l+1)}) \tnsr (x^*_{S(1,1)} \wdg \cdots x^*_{S(k+1,1))}) \tnsr \cdots (x^*_{S(1,l)} \wdg \cdots x^*_{S(k+1,l)})
\]

This element under the composition above goes to

\begin{eqnarray*}
\underset{s_1, \cdots s_l=1}{\overset{k+1}{\sum}} \prod^l_{j=1} (-1)^{s_1 + \cdots s_l}\det 
\begin{bmatrix}
    e_{T(1,j)} \tnsr x^*_{S(1,j)} & \cdots & \widehat{e_{T(1,j)} \tnsr x^*_{S(s_j,j)}} & \cdots & e_{T(1,j)} \tnsr x^*_{S(k+1,j)}\\

    &  & \vdots &  & \\
   e_{T(k,j)} \tnsr x^*_{S(1,j)} & \cdots & \widehat{e_{T(k,j)} \tnsr x^*_{S(s_j,j)}} & \cdots & e_{T(k,j)} \tnsr x^*_{S(k+1,j)}\\
  \end{bmatrix} \tnsr \\
  (e_{T(1,l+1)} \wdg \cdots e_{T(k,l+1)}) \tnsr (x^*_{S(s_1,1)} \cdots x^*_{S(s_l,l)})  
\end{eqnarray*}

So the image of $\rho$ is has a basis consisting of the above expressions for each choice of the basis element $(T,S)$.

\begin{lemma}
\label{lem:equiv}
In the notation above $\R^qp_* \scrptO_Z(i) \niso H^{-k}D^{\bullet}(i) \iso <\rho>/<\rho> \cap N$ as $H$-equivariant modules.
\end{lemma}

\begin{proof}

From the Remark~\ref{rem} and by induction on degree of $<\rho>/<\rho> \cap N$ we see that as graded $H$-representations we have $<\rho>/N \niso \R^qp_* \scrptO_Z(i)$. Moreover, for each representation which occurs in $<\rho>/N$ its embedding into $R \tnsr (\ExtAlg^k V \tnsr \Sym_l W^*)$ is given by the labeled diagram where $k$ "$\ExtAlg$"'s are added to the first $k$ rows in the $\GL(V)$-part and $l$ "$S$"'s are added to the first $l$ columns in the $\GL(W^*)$-part. Here the point is that in the quotient $<\rho>/<\rho> \cap N$ we never have to "slide up the "$V$"", so all the "$\ExtAlg$"'s remain in the same rows. Also the $\GL(W^*)$-part of each representation in $<\rho>/<\rho> \cap N$ is determined by it's $\GL(V)$-part exactly in the same way as $\delta(\lambda)$ is determined by $\lambda$ at the end of the Example~\ref{ex:coh}.

Let $Z$ and $B$ denote the modules of cycles and boundaries respectively in $D^{-k}(i)$. Then since $\rho \in H^{-k}D(i)$ and $\rho$ occurs only once in $D^{-k}(i)$ it must be that $<\rho> \subset Z$. Since $H^{-k}D(i)=Z/B$ does not contain any representations from $<\rho> \cap N$ it must be that $<\rho> \cap N \subset B$. Thus we have the following diagram:

\begin{figure}[h!]
\begin{tikzpicture}
\matrix(m) [matrix of math nodes, 
row sep=4.0em, column sep=4.0em, 
text height=1.5ex, text depth=0.25ex]
{<\rho>/<\rho> \cap N & H^{-k}D^{\bullet}(i)\\
<\rho> & Z \\
<\rho> \cap N &  B\\};

\path[->,font=\scriptsize] 
(m-1-1) edge node[above]{$\niso$}(m-1-2)
(m-2-1) edge (m-2-2)
(m-3-1) edge (m-3-2)
(m-2-1) edge (m-1-1)
(m-3-1) edge (m-2-1)
(m-2-2) edge (m-1-2)
(m-3-2) edge (m-2-2);
\end{tikzpicture}
\caption{}

\end{figure}
\FloatBarrier

Since $<\rho>/<\rho> \cap N \iso H^{-k}D^{\bullet}(i)$ as graded $H$-representations, they must be isomorphic as $H$-equivariant modules.  
\end{proof}

Note that we are not claiming that $Z=<\rho>$ and $B=<\rho> \cap N$. In fact it is not true by the degree considerations.

In fact, the $H$-equivariant structure of $\R^qp_* \scrptO_Z(i) \niso H^{-k}D^{\bullet}(i)$ can be described in a more compact way (this follows immediately from Lemma~\ref{lem:equiv}):

\begin{theorem}
\label{thm:lattice}

For given $i,q,f$ and $g$, let $k:=-q+g-1$ and $l:=-i-q-1$. Then the lattice of $\R^qp_* \scrptO_Z(i) \niso H^{-k}D^{\bullet}(i)$ as an $H$-equivariant module is obtained by taking the lattice of $I_{(l)^k}/I_{(l+1)^{k+1}}$ and adding to every representation one box to the end of each of the first $k$ rows of the $\GL(V)$-part and one box at the end of each of the first $l$ columns of the $\GL(W^*)$-part.
 
\end{theorem}

Here we should mention the result in \cite{Raic-Wey} where the terms of minimal resolution of $I_{(l)^k}$ were calculated.

\subsection{Lower Bound on the Projective Dimension of the Cohomology of $D^{\bullet}(i)$}

We can use the fact that the cohomology modules of $D^{\bullet}(i)$ are generated by one representation to obtain some information about their minimal $R$-resolutions. The terms in the minimal resolution of $H^{-k}D^{\bullet}(i)$ are obtained by calculating $\Tor^R_{\bullet}(H^{-k}D^{\bullet}(i), \Bbbk)$. Thus, let us take the tautological Koszul complex $K^{\bullet}$ which resolves $\Bbbk$ over $R$ (note that this is actually a complex of projectives in the category of $H$-equivariant modules):

\begin{figure}[h!]

\begin{tikzpicture}
\matrix(m) [matrix of math nodes, 
row sep=0.5em, column sep=0.5em, 
text height=1.5ex, text depth=0.25ex]
{K^{\bullet} \colon 0 & \underset{-fg}{R \underset{\Bbbk}{\tnsr} \ExtAlg^{fg}(V \tnsr W^*)} & R \underset{\Bbbk}{\tnsr} \ExtAlg^{fg-1} (V \tnsr W^*) & \cdots & R \underset{\Bbbk}{\tnsr} \ExtAlg^2 (V \tnsr W^*) & R \underset{\Bbbk}{\tnsr} (V \tnsr W^*) & \underset{0}{R} & 0\\};

\path[->,font=\scriptsize] 
(m-1-1) edge (m-1-2)
(m-1-2) edge (m-1-3) 
(m-1-3) edge (m-1-4) 
(m-1-4) edge (m-1-5)
(m-1-5) edge (m-1-6)
(m-1-6) edge (m-1-7)
(m-1-7) edge (m-1-8);
\end{tikzpicture}
\end{figure}
\FloatBarrier

After tensoring (over $R$) with $H^{-k}D^{\bullet}(i)$ we get the complex:

\begin{equation}
\begin{aligned}
H^{-k}D^{\bullet}(i) \underset{R}{\tnsr} K^{\bullet} \colon 0 \to \underset{-fg}{H^{-k}D^{\bullet}(i)\underset{\Bbbk}{\tnsr} \ExtAlg^{fg}(V \tnsr W^*)} \to H^{-k}D^{\bullet}(i)\underset{\Bbbk}{\tnsr} \ExtAlg^{fg-1} (V \tnsr W^*) \to \cdots \\
\to H^{-k}D^{\bullet}(i)\underset{\Bbbk}{\tnsr} \ExtAlg^2 (V \tnsr W^*) \to H^{-k}D^{\bullet}(i)\underset{\Bbbk}{\tnsr} (V \tnsr W^*) \to \underset{0}{H^{-k}D^{\bullet}(i)} \to 0
\end{aligned}
\end{equation}

Now, by merely combinatorial considerations, we obtain the following bound on the projective dimension of $H^{-k}D^{\bullet}(i)$:

\begin{proposition}
\label{prop:ProjDim}

The projective dimension of $\R^qp_* \scrptO_Z(i) \niso H^{-k}D^{\bullet}(i)$ is at least 

\[
(q+f-g+1)(min(-i,g)-q-1).
\]
\end{proposition}

\begin{proof}
The differentials in the complex $H^{-k}D^{\bullet}(i) \underset{R}{\tnsr} K^{\bullet}$ are $H$-equivariant (in particular, they map each representation to another representation) and preserve the degree. Let us look at the term $H^{-k}D^{\bullet}(i)\underset{\Bbbk}{\tnsr} \ExtAlg^j (V \tnsr W^*)$ in cohomological degree $-j$ of the complex $H^{-k}D^{\bullet}(i) \underset{R}{\tnsr} K^{\bullet}$. The image of previous differential $d^{-j-1}$ only starts from the next-to-the-lowest degree of $H^{-k}D^{\bullet}(i)\underset{\Bbbk}{\tnsr} \ExtAlg^j (V \tnsr W^*)$, thus every representation occuring in the lowest degree of $H^{-k}D^{\bullet}(i)\underset{\Bbbk}{\tnsr} \ExtAlg^j (V \tnsr W^*)$ which is in the kernel of the differential $d^{-j}$ will contribute to $\Tor^R_j(H^{-k}D^{\bullet}(i), \Bbbk)$. The representations which occur in the lowest degree of $H^{-k}D^{\bullet}(i)\underset{\Bbbk}{\tnsr} \ExtAlg^j (V \tnsr W^*)$, but do not occur in the next-to-the-lowest degree of $H^{-k}D^{\bullet}(i)\underset{\Bbbk}{\tnsr} \ExtAlg^{j-1} (V \tnsr W^*)$ have nowhere to go under the differential $d^{-j}$.

By using the Cauchy formula for exterior powers (\cite{Fult-Harr}, Section 6.1) we decompose the lowest degree of $H^{-k}D^{\bullet}(i)\underset{\Bbbk}{\tnsr} \ExtAlg^j (V \tnsr W^*)$ as follows:

\begin{equation}
\label{eq:decomp1}
(\bbS_{(l+1)^k}V \tnsr \bbS_{(l)^{k+1}}W^*) \tnsr \ExtAlg^j (V \tnsr W^*) \niso (\bbS_{(l+1)^k}V \tnsr \bbS_{(l)^{k+1}}W^*) \tnsr \underset{|\mu |=j}{\Drsum} \bbS_{\mu}V \tnsr \bbS_{\tilde{\mu}}W^*
\end{equation}

The next-to-the-lowest degree of $H^{-k}D^{\bullet}(i)\underset{\Bbbk}{\tnsr} \ExtAlg^{j-1} (V \tnsr W^*)$ is the tensor product of next-to-the-lowest degree of $H^{-k}D^{\bullet}(i)$ which is $\bbS_{(l+2, (l+1)^{k-1})}V \tnsr \bbS_{(l+1, l^k)}W^* \drsum \bbS_{((l+1)^k,1)}V \tnsr \bbS_{(l^{k+1},1)}W^*$ with $\ExtAlg^{j-1}(V \tnsr W^*)$:

\begin{equation}
\label{eq:decomp2}
[(\bbS_{(l+2, (l+1)^{k-1})}V \tnsr \bbS_{(l+1, l^k)}W^*) \tnsr \underset{|\mu |=j-1}{\Drsum} \bbS_{\mu}V \tnsr \bbS_{\tilde{\mu}}W^*] \drsum [(\bbS_{((l+1)^k,1)}V \tnsr \bbS_{(l^{k+1},1)}W^*) \tnsr \underset{|\mu |=j-1}{\Drsum} \bbS_{\mu}V \tnsr \bbS_{\tilde{\mu}}W^*]
\end{equation}

A partition $\mu$ of $j$ creates the summands $\bbS_{\alpha}V \tnsr \bbS_{\beta}W^*$ in decomposition~\ref{eq:decomp1} where $\bbS_{\alpha}V$ is a summand of $\bbS_{(l+1)^k}V \tnsr \bbS_{\mu} V$ and $\bbS_{\beta}W^*$ is a summand of $\bbS_{(l)^{k+1}}W^* \tnsr \bbS_{\tilde{\mu}}$. Suppose we want to have a summand in decomposition~\ref{eq:decomp1} where all the new boxes are added to $\bbS_{(l+1)^k}V \tnsr \bbS_{(l)^{k+1}}W^*$ only inside the grey region in the diagram below:

\begin{figure}[h!]
\begin{tikzpicture}[scale=0.5]
\draw (-11,0) -- (-4,0);
\draw (-4,0) -- (-4, 3);
\draw (-4,3) -- (-11,3);
\draw (-11,3) -- (-11,0);
\draw (-4,3) -- (0,3);
\draw (0,3) -- (0,0);
\draw (0,0) -- (-4,0);
\draw (-4,0) -- (-4,3);
\draw[fill=gray]  (0,0) -- (0,3) -- (-5,3) -- (-5,0)-- cycle;

\node[above] at (-8,3){\tiny{$l+1$}};
\node[left] at (-11,1.5){\tiny{$k$}};

\node at (1,1.5){\tiny{$\tnsr$}};

\draw (2,-1) -- (2,3) -- (8,3) -- (8,-1)-- cycle;
\draw (2,-1) -- (2,-6) -- (5,-6) -- (5,-1)-- cycle;
\draw[fill=gray] (2,-1) -- (2,-6) -- (5,-6) -- (5,-1)-- cycle;

\node[above] at (5,3){\tiny{$l$}};
\node[right] at (8,1){\tiny{$k+1$}};

\end{tikzpicture}
\caption{}
\end{figure}
\FloatBarrier

For that to happen, the necessary and sufficient conditions for $\mu$ are: $\height \mu \leq k$, $\height \mu =\tilde{\mu}_1 \leq l$ and $\mu_1=\height \tilde{\mu} \leq g-1-k$. As long as $j \leq min(l,k)(g-1-k)$ we have such a $\mu$.

Similarly, suppose we whant to make a summand in decomposition~\ref{eq:decomp1} where all the new boxes are added to $\bbS_{(l+1)^k}V \tnsr \bbS_{(l)^{k+1}}W^*$ only inside the grey region as in the diagram below: 

\begin{figure}[h!]
\begin{tikzpicture}[scale=0.5]

\draw (0,0) -- (0,3) -- (-7,3) -- (-7,0)-- cycle;
\draw (-7,0) -- (-3,0) -- (-3,-6) -- (-7,-6)-- cycle;
\draw[fill=gray] (-7,0) -- (-3,0) -- (-3,-6) -- (-7,-6)-- cycle;

\node[above] at (-4,3){\tiny{$l+1$}};
\node[left] at (-7,1.5){\tiny{$k$}};

\node at (1,1.5){\tiny{$\tnsr$}};

\draw (2,-1) -- (2,3) -- (8,3) -- (8,-1)-- cycle;
\draw (8,-1) -- (14,-1) -- (14,3) -- (8,3)-- cycle;
\draw[fill=gray] (8,-1) -- (14,-1) -- (14,3) -- (8,3)-- cycle;

\node[above] at (5,3){\tiny{$l$}};
\node[right] at (2,1.5){\tiny{$k+1$}};

\end{tikzpicture}
\caption{}
\end{figure}
\FloatBarrier

For that to happen, the necessary and sufficient conditions for $\mu$ are: $\mu_1 \leq l+1$, $\mu_1=\height \tilde{\mu} \leq k+1$ and $\height \mu \leq f-k$. As long as $j \leq min(l+1,k+1)(f-k)$ we have such a $\mu$.

Thus, if $j \leq max(min(l,k)(g-1-k), min(l,k)(f-k))=max(g-1-k, f-k)min(l,k)$ we always have representations in the next-to-the-lowest degree of $H^{-k}D^{\bullet}(i)\underset{\Bbbk}{\tnsr} \ExtAlg^j (V \tnsr W^*)$ which have nowhere to go under the differential $d^{-j}$ because every summand in the decomposition~\ref{eq:decomp2} has either

\begin{enumerate}
\item one box added to the first row of both $\GL(V)$- and $\GL(W^*)$-parts of $\bbS_{(l+1)^k}V \tnsr \bbS_{(l)^{k+1}}W^*$ (summands in the first square brackets in~\ref{eq:decomp2}), or\\
\item one box added to the first column of both $\GL(V)$- and $\GL(W^*)$-parts of $\bbS_{(l+1)^k}V \tnsr \bbS_{(l)^{k+1}}W^*$ (summands in the second square brackets in~\ref{eq:decomp2}).
\end{enumerate}

Substituting $k=-q+g-1$, $l=-i-q-1$ we simplify the expression $max(g-1-k, f-k)min(l,k)=max(q, q+f-g+1)min(-i-q-1,-q+g-1)=(q+f-g+1)(min(-i,g)-q-1)$
\end{proof}

\section{Tilting Theory Perspective}
\label{sec:Tilt}

Let us briefly recall some notions from tilting theory in algebraic geometry. By $X$ we will denote any variety over $\Spec \Bbbk$ with the structure map $f \colon X \to \Spec \Bbbk$. Let $\D^b(X)$ denote the bounded derived category of coherent sheaves on $X$. (see Chapter 3 of~\cite{Huy}). The following definitions are from \cite{Huy}:

\begin{definition}
An object $E \in \D^b(X)$ is called exceptional if
\[
\Hom_{D^b(X)}(E, E[l]) \niso
\begin{cases}
\Bbbk \text{ if } l=0 \\
0 \text{ otherwise}
\end{cases}
\]

An (ordered) sequence $(E_1, \cdots E_n)$ of exceptional objects is called an exceptional sequence if $\Hom_{\D^b(X)}(E_i, E_j[l])=0$ for all $i>j$ and all $l$.
 
An exceptional sequence $(E_1, \cdots E_n)$ is called strong if it satisfies $\Hom_{\D^b(X)}(E_i, E_j[l])=0$ for all $i,j$ and $l \neq 0$.

An exceptional sequence $(E_1, \cdots E_n)$ is called full if $\D^b(X)$ is generated by $\{E_1, \cdots E_n \}$, i.e. any full triangualted subcategory containing the objects $\{E_1, \cdots E_n \}$ is equivalent to $\D^b(X)$ via the inclusion.

\end{definition}

The importance of full strong exceptional sequences is formulated in the following theorem (see Theorem 7.6 in Fourier-Mukai Transforms,~\cite{Hand}):

\begin{theorem}
\label{thm:tilt}

Let $X$ be projective over a Noetherian affine scheme of finite type. Suppose there exists a full strong exceptional sequence $(E_1, \cdots E_n)$ in $\D^b(X)$. Let $E:=\underset{i}{\Drsum} E_i$ and let $A:= \End_{\D^b(X)}(E)$, then 

\[
\R\Hom(E, \_) \colon \D^b(X) \to \D^b(\fmod-A)
\]

is an exact equivalence.
\end{theorem}

Even though in the theorem above one needs the exceptional sequence to be strong, the exceptional sequences which are not strong are quite important. Existence of exceptional sequence in $\D^b(X)$ gives a sequence of semi-orthogonal subcategories in $\D^b(X)$ and if the exceptional sequence is full then one gets a semi-orthogonal decomposition of $\D^b(X)$ (Example 1.60 in~\cite{Huy}). A good analogy for semi-orthogonal sequence of subcategories is a sequence of orthogonal subspaces of a vectror space. In this analogy full exceptional sequence corresponds to orthogonal basis of a vector space. Also see Theorem~\ref{thm:mut} below. 

\begin{example}
The content of this example is a classical result by Beilinson(~\cite{Bei}). Let $X=\bbP(W^*)$ (recall that the dimension of $\bbP$ is $g-1$). Then for any $a \in \bbZ$ the sequence $(\scrptO(a), \scrptO(a+1), \cdots \scrptO(a+g-1))$ is a strong full exceptional. Also the sequence $(\scrptO, \Omega(1), \ExtAlg^2 \Omega(2), \ExtAlg^{g-1} \Omega(g-1))$ is a strong full exceptional sequence.

Let $E= \overset{i=g-1}{\underset{i=0}{\Drsum}} \ExtAlg^i \Omega(i)$, then $A:=\End_{\D^b(X)}(E)$ is the upper triangular matrix ring

\begin{eqnarray*}
A \niso 
\begin{bmatrix}
    \ExtAlg^0 W & \ExtAlg^1 W& \cdots & \ExtAlg^{g-1} W\\
   0& \ExtAlg^1 W& \cdots & \ExtAlg^{g-2} W \\
    
    &  & \vdots &  & \\
   0 & 0 & \cdots & \ExtAlg^0 W\\
\end{bmatrix}
\end{eqnarray*}

Thus by Theorem~\ref{thm:tilt} $\D^b(\bbP(W^*)) \niso \D^b(\fmod-A)$. If in Theorem~\ref{thm:tilt} we put $E'=\overset{i=g-1}{\underset{i=0}{\Drsum}} \scrptO(i)$ then $A'=\End_{\D^b(X)}(E')$ is the lower triangular matrix ring

\begin{eqnarray*}
A' \niso 
\begin{bmatrix}
    \Sym_0 W^* & 0 & \cdots & 0\\
   \Sym_1 W^*& \Sym_0 W^* & \cdots & 0 \\
    
    &  & \vdots &  & \\
    \Sym_{g-1} W^* & \Sym_{g-2} W^* & \cdots & \Sym_0 W^*\\
\end{bmatrix}
\end{eqnarray*}

So, by Theorem~\ref{thm:tilt} we also have that $\D^b(\bbP(W^*)) \niso \D^b(\fmod-A')$. 
\end{example}

The two sequences in the above example are closesly related by so-called mutations. Here we follow the exposition from~\cite{Bohn}. For $A, B \in \D^b(X)$ set $\Hom^{\times}(A,B):=\underset{k \in \bbZ}{\Drsum} \Ext^k (A,B)$ - a graded finite-dimensional $\Bbbk$-vector space. For $Y \in \D^b(X)$ and $U \niso \underset{k \in \bbZ}{\Drsum}U_k$ any graded $\Bbbk$-vector space we define $Y \tnsr U:= \underset{k \in \bbZ}{\Drsum} Y \tnsr U_k [-k]$ where $Y \tnsr U_k$ means the complex $Y \tnsr f^* U_k$. 

\begin{definition}
\label{def:mut}

Let $(E_1, E_2)$ be an exceptional sequence in $\D^b(X)$. The object $L_{E_1}E_2 \in \D^b(X)$ is called the left mutation of $E_2$ across $E_1$ and is defined by the distinguished triangle:

\[
L_{E_1}E_2 \to \Hom^{\times}(E_1, E_2) \tnsr E_1 \to E_2 \to L_{E_1}E_2
\]

The map $\Hom^{\times}(E_1, E_2) \tnsr E_1 \to E_2$ is the canonical evaluation map.

Similarly, the object $R_{E_2}E_1 \in \D^b(X)$ is called the right mutation of $E_1$ across $E_2$ and is defined by the distinguished triangle:

\[
R_{E_2}E_1[-1] \to E_1 \to \Hom^{\times}(E_1, E_2)^* \tnsr E_2 \to R_{E_2}E_1
\]
 
\end{definition}

\begin{theorem}
\label{thm:mut}
Let $(E_1, \cdots E_n)$ be an exceptional sequence in $\D^b(X)$. Set

\[
R_i(E_1, \cdots E_n):= (E_1, \cdots E_{i-1}, E_{i+1}, R_{E_{i+1}}E_i, E_{i+2}, \cdots E_n)
\]

and

\[
L_i(E_1, \cdots E_n):= (E_1, \cdots E_{i-1}, L_{E_i}E_{i+1}, E_i, E_{i+2},  \cdots E_n)
\]

Then $R_i(E_1, \cdots E_n)$ and $L_i(E_1, \cdots E_n)$ are again exceptional sequences. If $(E_1, \cdots E_n)$ is full, then so are $R_i(E_1, \cdots E_n)$ and $L_i(E_1, \cdots E_n)$. Moreover, $R_i$ and $L_i$ are inverse operations and they define the action of Braid group on $n$ strings on the set of exceptional sequences with $n$ terms in $\D^b(X)$. If the sequence $(E_1, \cdots E_n)$ is strong, then  neither $L_i(E_1, \cdots E_n)$ or $R_i(E_1, \cdots E_n)$ is in general strong.

\end{theorem}

\begin{example}
Consider the (full, strong) exceptional sequece $(\ExtAlg^{g-1}\Omega(g-1), \ExtAlg^{g-2}\Omega(g-2), \cdots \scrptO)$. Then $R_{g-1} (\ExtAlg^{g-1}\Omega(g-1), \ExtAlg^{g-2}\Omega(g-2), \cdots \scrptO)= (\ExtAlg^{g-1}\Omega(g-1), \ExtAlg^{g-2}\Omega(g-2), \cdots \ExtAlg^2 \Omega (2), \scrptO, R_{\scrptO} \ExtAlg^1 \Omega (1))$. To calculate $R_{\scrptO} \ExtAlg^1 \Omega (1)$ we use the Euler sequence:

\[
0 \to \Omega(1) \to \pi^* W \to \scrptO(1) \to 0
\]

One can calculate using Theorem~\ref{thm:BBW} that $\R^0 \pi_* \scrptH om(\ExtAlg^j (\Omega (1)), \scrptO) \niso \ExtAlg^j W^*$ and the map $\Omega(1) \to \pi^* W$ is the same as in Definition~\ref{def:mut} applied to $R_{\scrptO} \ExtAlg^1 \Omega (1)$. Thus $R_{\scrptO} \ExtAlg^1 \Omega (1)= \scrptO(1)$ and so $(\ExtAlg^{g-1}\Omega(g-1), \ExtAlg^{g-2}\Omega(g-2), \cdots \ExtAlg^2 \Omega (2), \scrptO, R_{\scrptO} \ExtAlg^1 \Omega (1)) = (\ExtAlg^{g-1}\Omega(g-1), \ExtAlg^{g-2}\Omega(g-2), \cdots \ExtAlg^2 \Omega (2), \scrptO, \scrptO(1))$.

Similarly, one can use the exact sequence 

\[
0 \to \ExtAlg^2 \Omega (2) \to \pi^* \ExtAlg^2 W \to \Omega(2) \to 0
\] 

to see that $R_{\scrptO} \ExtAlg^2 \Omega(2) = \Omega(2)$. So we see that one can start from the sequence $(\ExtAlg^{g-1}\Omega(g-1), \ExtAlg^{g-2}\Omega(g-2), \cdots \scrptO)$ and get to the sequence $(\scrptO, \scrptO(1), \cdots \scrptO(g-1))$ by a series of right mutations. At intermediate steps we get exceptional sequences of the form $(\ExtAlg^{g-1} \Omega (g-1), \cdots \ExtAlg^l \Omega(l), \scrptO, \cdots \scrptO(l-1))$.
 
\end{example}

The strong full exceptional sequence $(\scrptO, \Omega(1), \ExtAlg^2 \Omega(2), \ExtAlg^{g-1} \Omega(g-1))$ on $\bbP(W^*)$ remains such after being lifted to the total space of the vector bundle $Z$, i.e. the sequence $((\pi \iota)^*\scrptO, (\pi \iota)^* \Omega(1), (\pi \iota)^* \ExtAlg^2 \Omega(2), (\pi \iota)^* \ExtAlg^{g-1} \Omega(g-1))$ is a strong full exceptional sequence on $Z$. This is one of the results (Theorem C) in ~\cite{Buch-Leu-Van-den-Bergh}.

On the other hand, if we lift the exceptional sequence $(\scrptO_{\bbP(W^*)}, \cdots \scrptO_{\bbP(W^*)}(g-1))$ from $\bbP(W^*)$ to $Z$ we get the sequence $(\scrptO_Z, \cdots \scrptO_Z(g-1))$. For $i \geq j$ let us calculate $\Hom_{\D^b(Z)}(\scrptO_Z(i), \scrptO_Z(j)[l])$:

\[
\Hom_{\D^b(Z)}(\scrptO_Z(i), \scrptO_Z(j)) \niso \Hom_{\D^b(Z)}(\scrptO, \scrptO(j-i)[l]) \niso \R^l p_* \scrptO_Z(j-i)
\]

The cohomology $R^l p_* \scrptO_Z(j-i)$ is computed by the complex $D^{\bullet}(j-i)$, so we see that each $\scrptO_Z(l)$ remains exceptional, but the sequence $(\scrptO_Z, \cdots \scrptO_Z(g-1))$ is not exceptional, and the obstruction for it to be exceptional is precisely the non-exactness of the complexes $D^{\bullet}(l)$ for $l<0$.

As we saw above, the sequence $(\ExtAlg^{g-1}\Omega(g-1), \ExtAlg^{g-2}\Omega(g-2), \cdots \scrptO)$ (whose lifting to $Z$ remains exceptional) can be transformed to $(\scrptO, \cdots \scrptO(g-1))$ through a series of right mutations. At intermediate steps we have sequences of the form $(\ExtAlg^{g-1} \Omega (g-1), \cdots \ExtAlg^l \Omega(l), \scrptO, \cdots \scrptO(l-1))$. Checking whether lifting of such sequence is an exceptional sequence ammounts to computation of 

\begin{equation}
\begin{aligned}
\label{eq:excseq}
\Hom_{\D^b(Z)}(\scrptO_Z(s), (\pi \iota)^* \ExtAlg^k \Omega(k)[l]) \niso \Hom_{\D^b(Z)}(\scrptO_Z, (\pi \iota)^*\ExtAlg^k \Omega(k-s)[l]) \niso \\
\R^l (p \iota) (\iota^* \pi^* \ExtAlg^k \Omega(k-s)) \niso \R^lp_*(\pi^* \ExtAlg^k \Omega(k-s) \tnsr \iota_*\scrptO_Z) 
\end{aligned}
\end{equation}

Using commutativity of the diagram~\ref{fig:setup} we can go in counter-clockwise direction to compute:

\begin{eqnarray*}
\R^l(\pi'p)_*(\pi^* \ExtAlg^k \Omega(i) \tnsr \iota_* \scrptO_Z) \niso \R^lp'_* (\pi_* (\pi^* \ExtAlg^k \Omega(i) \tnsr \iota_* \scrptO_Z)) \niso \R^lp'_*(\ExtAlg^k \Omega(i) \tnsr \pi_* \iota_* \scrptO_Z) \niso \\
\R^lp'_*(\ExtAlg^k \Omega(i) \tnsr \Sym (\scrptT(-1) \tnsr p'^*V)) \niso \underset{\lambda}{\Drsum}(\R^lp'_*[\ExtAlg^k \Omega(i) \tnsr \bbS_{\lambda} \scrptT(-1)] \tnsr \bbS_{\lambda}V)
\end{eqnarray*}

For each $\lambda=(\lambda_1, \cdots \lambda_{g-1})$ we have:

\[
\ExtAlg^k \Omega(i) \tnsr \bbS_{(\lambda_1, \cdots \lambda_{g-1})} \scrptT(-1) \niso \bbS_{(1^k)}(\Omega(1)) \tnsr \bbS_{(-\lambda_{g-1}, \cdots -\lambda_1)}(\Omega(1)) \tnsr \scrptO(i-k)
\]

Let $\lambda':=(\lambda'_1, \cdots \lambda'_{g-1}) \in VS(\lambda, k)$. Then in order for this $\lambda'$ to contribute to $\R^qp'_*[\ExtAlg^k \Omega(i) \tnsr \bbS_{\lambda} \scrptT(-1)]$ we see as in Example~\ref{ex:coh}, that $\lambda'$ must satisfy the inequalities:

\[
\lambda'_{-q+g}<-d-q \leq \lambda'_{-q+g-1}
\]

where $d:=i-k$. In particular, we see that for each $d=i-k \leq -1$, there exists $0 \leq q \leq g-1$ such that $\R^l(\pi'p)_*(\pi^* \ExtAlg^k \Omega(i) \tnsr \iota_* \scrptO_Z)$ is non-zero, which means (see Equation~\ref{eq:excseq}) that none of the exceptional sequences of the form $(\ExtAlg^{g-1} \Omega (g-1), \cdots \ExtAlg^l \Omega(l), \scrptO, \cdots \scrptO(l-1))$ on $\bbP(W^*)$ lift to an exceptional sequence on $Z$.

\section{Complexes $D^{\bullet}(i,k)$}

Fix $k, i \in \bbZ, g-1 \geq k \geq 0$. Suppose we want to compute $\R^{\bullet} (p \iota)_* ((\pi \iota)^* \ExtAlg^k \Omega(i))$. As with the exceptional sequence $(\scrptO, \cdots \scrptO(g-1))$, we can construct a family of complexes which will compute the desired cohomology.

We use the projection formula:

\[
\R^{\bullet} (p \iota)_* ((\pi \iota)^* \ExtAlg^k \Omega(i)) \niso \R^{\bullet}p_* \R^{\bullet} \iota_*(\iota^* \pi^* \ExtAlg^k \Omega(i)) \niso \R^{\bullet}p_* (\pi^* \ExtAlg^k \Omega(i) \tnsr \R^{\bullet} \iota_* \scrptO_Z)
\]

But in $\D^b(\bbP)$ we can replace  $\R^{\bullet} \iota_* \scrptO_Z$ by its resolution $\scrptK^{\bullet}$. Thus $\R^{\bullet}p_* (\pi^* \ExtAlg^k \Omega(i) \tnsr \R^{\bullet} \iota_* \scrptO_Z)$, can be computed as $\R^{\bullet}p_*(\_)$ of the complex $\scrptK^{\bullet} \tnsr \pi^* \ExtAlg^k \Omega(i)$. We proceed as in Section~\ref{sec:GenComp}, namely we choose a Cartan-Eilenberg resolution of $\scrptK^{\bullet} \tnsr \pi^* \ExtAlg^k \Omega(i)$, apply $p_*$ to it and consider two spectral sequences associated to this double complex. The spectral sequence $_{\bullet}^{row}E^{\bullet, \bullet}$ converges to $\R^{\bullet} (p \iota)_* ((\pi \iota)^* \ExtAlg^k \Omega(i))$. Let us look at $_1^{col}E^{\bullet, \bullet}$. We have: 

\[
_1^{col}E^{p,q} = \R^qp_*(\scrptK^p \tnsr \pi^* \ExtAlg^k \Omega(i)) \niso  \R^qp_*(p^* \ExtAlg^{-p}F(p) \tnsr \pi^* \ExtAlg^k  \Omega(i)) \niso \ExtAlg^{-p}F \tnsr \R^qp_* \pi^* \ExtAlg^k \Omega(i+p)
\]

where the last isomorphism is the projection formula. Because cohomology commutes with flat base change, we can simplify further:

\[
\ExtAlg^{-p}F \tnsr \R^qp_* \pi^* \ExtAlg^k \Omega(i+p) \niso \ExtAlg^{-p}F \tnsr \pi'^* \R^qp'_* \ExtAlg^k \Omega (i+p)
\]

For convenience let $d:=i-k$. Now using Theorem~\ref{thm:BBW} we compute:

\[
\R^qp_* \ExtAlg^k \Omega(i+p) \niso
\begin{cases}
\bbS_{(d+p,1^k,0, \cdots 0)}W \text{ if $p\geq -d+1$ and $q=0$} \\
\Bbbk \text{ if $p=-k-d$ and $q=k$} \\
\bbS_{(-d-p-g+1,1 \cdots, 1, 0^k)}W^* \text{ if $p\leq -d-g$ and $q=g-1$} \\
0 \text{ otherwise} \\
\end{cases}
\]

As in Section~\ref{sec:GenComp} we draw the two spectral sequences superimposed on each other we get the following picture:

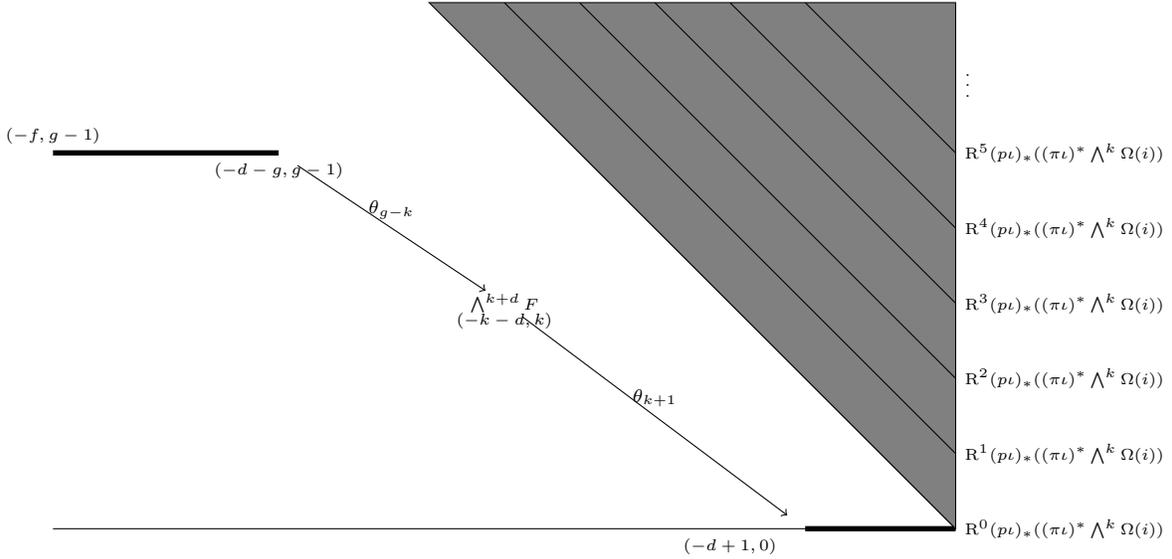
\begin{figure}[h!]
\begin{tikzpicture}[scale=1]
\draw (-12,0) -- (0,0);
\draw (0,0) -- (0, 7);
\draw (0,0) -- (-6,6);
\draw [line width=2] (-12, 5) -- (-9, 5);
\draw [line width=2] (-2,0) -- (0,0);
\node[above] at (-12,5){\tiny{$(-f,g-1)$}};
\node[below] at (-9,5){\tiny{$(-d-g,g-1)$}};
\node[below] at (-6,3){\tiny{$(-k-d,k)$}};
\node at (-6,3){\tiny{$\ExtAlg^{k+d}F$}};
\node[below] at (-3,0){\tiny{$(-d+1,0)$}};
\node[right] at (0,0){\tiny{$\R^0 (p \iota)_* ((\pi \iota)^* \ExtAlg^k \Omega(i))$}};
\node[right] at (0,1){\tiny{$\R^1 (p \iota)_* ((\pi \iota)^* \ExtAlg^k \Omega(i))$}};
\node[right] at (0,2){\tiny{$\R^2 (p \iota)_* ((\pi \iota)^* \ExtAlg^k \Omega(i))$}};
\node[right] at (0,3){\tiny{$\R^3 (p \iota)_* ((\pi \iota)^* \ExtAlg^k \Omega(i))$}};
\node[right] at (0,4){\tiny{$\R^4 (p \iota)_* ((\pi \iota)^* \ExtAlg^k \Omega(i))$}};
\node[right] at (0,5){\tiny{$\R^5 (p \iota)_* ((\pi \iota)^* \ExtAlg^k \Omega(i))$}};
\node[right] at (0,6){\tiny{$\vdots$}};
\draw[fill=gray]  (0,0) -- (-7,7) -- (0,7) -- cycle;
\draw  (0,1) -- (-6,7);
\draw  (0,2) -- (-5,7);
\draw  (0,3) -- (-4,7);
\draw  (0,4) -- (-3,7);
\draw  (0,5) -- (-2,7);
\path[->,font=\scriptsize] 
(-9,5) edge[shorten <= .3cm, shorten >= .3cm]node[above]{$\theta_{g-k}$} (-6,3)
(-6,3) edge[shorten <= .3cm, shorten >= .3cm] node[above]{$\theta_{k+1}$} (-2,0);

\end{tikzpicture}
\caption{$f=12$, $g=6$, $d=3$, $k=3$}
\end{figure}
\FloatBarrier

Here $\theta_{g-k}$ and $\theta_{k+1}$ denote differentials on pages $g-k$ and $k+1$ respectively. 

\begin{definition}
Fix $k, i \in \bbZ, g-1 \geq k \geq 0$. Let $d:=i-k$. Define $D^{\bullet}(i,k)$ be the complex derived (as in Section~\ref{sec:GenComp}) from the above spectral sequence. 
\end{definition}

The acyclicity properties of these complexes can be proved just like in Section~\ref{sec:GenComp}. In particular, we see that only when $d=i-k<0$ we start getting intersections with the grey region, thus there is cohomology.

\end{document}